\documentclass[a4paper,12pt]{article}
\usepackage{amsmath,amsthm,amsfonts,amssymb,bbm}
\usepackage{graphicx,psfrag,subfigure,color,cite}

\numberwithin{equation}{section}

\usepackage[margin=1.2in]{geometry}

\newcommand{\id}{{\rm id}}
\newcommand{\e}{\varepsilon}
\newcommand{\Pb}{\mathbb{P}}
\newcommand{\E}{\mathbb{E}}

\newcommand{\R}{\mathbb{R}}
\newcommand{\N}{\mathbb{N}}
\newcommand{\Z}{\mathbb{Z}}

\newcommand{\eps}{\varepsilon}
\newcommand{\Or}{{\cal O}}

\newtheorem{prop}{Proposition}[section]
\newtheorem{assumption}[prop]{Assumption}
\newtheorem{thm}[prop]{Theorem}
\newtheorem{lem}[prop]{Lemma}

\newtheorem{cor}[prop]{Corollary}

\newtheorem{cla}[prop]{Claim}

\newtheorem{rem}[prop]{Remark}
\newenvironment{remark}{\begin{rem}\normalfont}{\end{rem}}

\title{TASEP with a moving wall}
\author{Alexei Borodin\thanks{Department of Mathematics, Massachusetts Institute of Technology, 77 Massachusetts Ave., Cambridge, MA 02139, USA. E-mail: {\tt borodin@math.mit.edu }}
\and Alexey Bufetov\thanks{Institute of Mathematics, Leipzig University, Augustusplatz 10, 04109 Leipzig, Germany. E-mail: {\tt alexey.bufetov@gmail.com}}
\and Patrik L.\ Ferrari\thanks{Institute for Applied Mathematics, Bonn University, Endenicher Allee 60, 53115 Bonn, Germany. E-mail: {\tt ferrari@uni-bonn.de}}}

\date{November 2, 2021}

\begin{document}
\maketitle
\begin{abstract}
We consider a totally asymmetric simple exclusion on $\mathbb{Z}$ with the
step initial condition, under the additional restriction that the first
particle cannot cross a deterministally moving wall. We prove that
such a wall may induce asymptotic fluctuation distributions of particle
positions of the form
$$
\Pb\Big(\sup_{\tau\in \R}\{\textrm{Airy}_2(\tau) -g(\tau)\}\leq S\Big)
$$
with arbitrary barrier functions $g$. This is the same class of distributions
that arises as one-point asymptotic fluctuations of TASEPs with arbitrary
initial conditions. Examples include Tracy-Widom GOE and GUE distributions, as
well as a crossover between them, all arising from various particles behind a
linearly moving wall.

We also prove that if the right-most particle is second class, and a linearly
moving wall is shock-inducing, then the asymptotic distribution of the
position of the second class particle is a mixture of the uniform distribution
on a segment and the atomic measure at its right end.
\end{abstract}

\sloppy

\section{Introduction} \label{SectIntro}

The totally asymmetric simple exclusion process (or \emph{TASEP}, for short) is
a prototypical example of an interacting particle system in one
space dimension. It consists of particles moving within the one-dimensional
lattice $\mathbb{Z}$ in continuous time, with each site of $\mathbb{Z}$
occupied by at most one particle (the \emph{exclusion} constrained). Each
particle carries its individual exponential clock of rate one (all clocks are
independent), and when that clocks rings, the particle attempts to jump to the
right by one unit. It succeeds if the target site is empty, and it stays put if
it is not. The TASEP was introduced into mathematics by Spitzer in~\cite{Spi70},
and since then it has been a subject of extensive studies.

The TASEP evolution can be rephrased as a model of random interface growth, and
as such it is arguably the simplest representative of the (conjectural)
Kardar-Parisi-Zhang universality class of random growth models in
(1+1)-dimensions~\cite{KPZ86}. For example, it was the first member of the
class for which one-point fluctuations were asymptotically analyzed (Johansson
\cite{Jo00b}), and the complete Markovian evolution of the asymptotic
fluctuation processes was obtained (Matetski-Quastel-Remenik~\cite{MQR17}).

Among all possible initial conditions for the TASEP, the translation-invariant
stationary ones play a special role. They depend on a single parameter
$\rho\in (0,1)$ called the \emph{density}, and they place a particle at each
site of $\mathbb{Z}$ independently with probability $\rho$. One shows that
such a Bernoulli measure on particle configurations in $\mathbb{Z}$ is stable
under the TASEP evolution, and the trajectory of each particle is a simple
random walk in continuous time that moves up by one with rate $(1-\rho)$.
\footnote{The latter fact is usually referred to as \emph{Burke's
theorem}, see Burke~\cite{Bur56} for the original statement in terms of a
queuing system, and Spitzer~\cite[Example 3.2]{Spi70}, Ferrari-Fontes
\cite{FF96} for interpretations in terms of particle systems.}

Apart from the stationary initial data, the simplest initial
configuration is probably the one that is often called the \emph{step initial
condition}: The particles occupy all the sites marked by $\mathbb{Z}_{\le
0}=\{0,-1,-2,\dots\}$. This was the one considered in~\cite{Jo00b}, and it was
also the first for which a (nontrivial) law of large numbers type behavior was
obtained by Rost in~\cite{R81}.

The purpose of the present paper is to investigate the fluctuations of the
TASEP with the step initial condition, with an additional constraint -- the
right-most particle is forbidden to cross a deterministically moving barrier
that we call \emph{a wall}.

Similar setups have been considered before. However, in previous works the
movement of the wall was always \emph{random}, and, to our knowledge, a
deterministically moving wall has not been considered before.

Borodin-Ferrari-Sasamoto~\cite{BFS09} studied the situation when the TASEP is
initialized with the particles occupying every second site in $\mathbb{Z}_{\le
0}$, and right-most particle has a different (slower) jump rate. Equivalently,
one can think of the positive semi-axis $\Z_{>0}$ initialized by a
Bernoulli measure. A few different fluctuation processes arose in the large
time limit, depending on the speed of the right-most particle and the region of
the lattice, cf.~\cite[Section 2]{BFS09}. One interesting feature was a
fluctuation description of a \emph{shock}. Shocks were investigated on a deeper
level in subsequent works by Ferrari-Nejjar~\cite{FN13, FN16, N17, Fer18, N19,
FN19} and Quastel-Rahman~\cite{QR18}. Other types of evolution of the first
particle were considered (and other initial condition on $\mathbb{Z}_{\ge 0}$ as
well), with fluctuations of different order and different distributions.
However, the randomness of the wall movement was always an essential
contributor to the fluctuations of the TASEP particles, once they became
affected by the wall.

Our main interest in this work was to investigate the situation when the
movement of the wall was not producing any randomness whatsoever, and to see
what kind of asymptotic fluctuations one would be able to observe in such a
case.

Our main result says that by choosing an appropriate (deterministic) movement
of the wall, one can create distributions of the asymptotic fluctuations of
a particle in the bulk of the system that span essentially the same class as
the ones arising from one-point asymptotic fluctuations of the TASEP with
varying (deterministic) initial conditions (obtained previously in
\cite{MQR17}). A bit more precisely, we show that for any piecewise continuous
function $g$ on the real line satisfying $g(\tau)\ge const+\tau^2/2$, $\tau\in
\mathbb{R}$, the probability that a certain
particle at time $t>0$ is to the right of position $\xi t- Sct^{1/3}$, for
suitable constants $\xi$ and $c$, tends, as $t\to\infty$, to
\begin{equation}
\label{eq:intro1}
\Pb\Big(\sup_{\tau\in \R}\{{\cal A}_2(\tau) -g(\tau)\}\leq S\Big), \qquad
S\in \mathbb{R},
\end{equation}
where ${\cal A}_2$ is the Airy$_2$ process. The corresponding movement of the
wall is governed by the function $g$ in a space-time window of
size $\sim t^{1/3}\times t^{2/3}$ which is determined by the particle that we
focus on; outside of that critical window the position of the wall is required
to satisfy an (explicit) macroscopic inequality that limits its influence on
the chosen particle.

An exact formulation of the claim above can be found
in Theorem \ref{ThmOnePointSituation} below.

The conditions on the functions $g$ can be relaxed, although we do not pursue
that in the present paper. Our result also does not cover all possible
fluctuation scenarios that a deterministic wall can create. For example, wall
trajectories may have multiple critical windows that affect the tagged
particle, and those are not considered in this text.

For a discussion of how the variational formula \eqref{eq:intro1} is
related to the \emph{Airy sheet} and the \emph{KPZ fixed point}, see
\cite[Section 4.5]{MQR17} and references therein. For some of the earlier
works where such variational formulas played an important role see Johansson
\cite{Jo03}, Quastel-Remenik~\cite{QR13, QR16}, Baik-Liu~\cite{BL13},
Corwin-Liu-Wang~\cite{CLW16}, Chhita-Ferrari-Spohn~\cite{CFS16},
Ferrari-Occelli~\cite{FO17}.

To give a concrete application, consider the wall that starts at the origin and
moves to the right with constant speed $v<1$. Then we show, in Section
\ref{secConstantSpeed} below, that our main result implies the following
behavior. Denote by $x_1(t)>x_2(t)>\dots$ the positions of our TASEP particles
at time $t\ge 0$; note that $x_n(0)=-n+1$ for all $n\ge 1$. Then the asymptotic
fluctuations of $x_{\alpha t}(t)$ converge, on a $t^{1/3}$-scale and as
$t\to\infty$, to the GOE Tracy-Widom distribution $\mathrm{F_1}$ for
$\alpha<(1-v)^2$, to the GUE Tracy-Widom distribution $\mathrm{F_2}$ for
$\alpha>(1-v)^2$, and to the crossover distribution $\mathrm{F_{2\to 1;0}}$
describing a section of the Airy$_{2\to 1}$ process of~\cite{BFS07} for
$\alpha=(1-v)^2$.

Our proof is based on two relatively recent advances.

One is a so-called \emph{color-position symmetry} of the multi-species TASEP.
The symmetry itself
goes back to the work of Angel-Holroyd-Romik~\cite{AHR09}, see also
Amir-Angel-Valk\'{o}~\cite{AAV11}, Borodin-Wheeler~\cite{BW18}, Borodin-Bufetov
\cite{BB19}, Bufetov~\cite{B20}, Galashin~\cite{G20} for the development of
its understanding and some of its applications. For our model with a wall,
we show that this symmetry implies that the distribution of the position of a
given particle is identical to the distribution of the position of another
particle in another TASEP with the step initial condition (and without a wall),
conditioned on the fact that this other particle remains ahead of a
deterministally moving barrier.

The second part of our argument is a precise control of the whole trajectory of
a given TASEP particle achieved via the technique of \emph{backwards paths}
introduced in~\cite{Fer18}, see also~\cite{FN19}. Backwards paths are random
lattice paths that mimic the behavior of the characteristics of
exclusion processes.\footnote{More exactly, those are the characteristics of
the inviscid Burgers equation that describes the law of large numbers behavior
of the exclusion processes.} The control is realized by fine comparisons of the
backwards paths for the process at hand with those of stationary TASEPs, since
the latter ones are easier to estimate. Backwards paths could be
viewed as analogs of geodesics in last passage percolation models.

Employing this technique for the trajectory of a TASEP particle conditioned to
stay above a barrier, which we obtain via the color-position symmetry from the
original TASEP with a wall, ultimately leads to the limiting distributions
\eqref{eq:intro1}, where the function $g$ is related to the movement of the
barrier and, consequently, to the movement of the wall in the original
TASEP.

We also offer another application of our approach by considering
a TASEP with a linearly moving wall whose right-most particle is second class.
Second class particles are known to track the characteristics and, in
particular, stick to the shock locations. For the step initial condition
without a wall, Ferrari-Kipnis~\cite{FK95} proved that at a large time $t$, the
right-most second class particle is asymptotically uniformly distributed on the
segment $(-t,t)$.\footnote{This corresponds to the fact that in the case of the
step initial condition, the characteristics that pass through the origin form a
\emph{rarefaction fan}.} We prove, in Theorem \ref{th:sec-class} below, that
once one adds a shock-inducing wall that at time $t\in [0,T]$ is at position
$cT+vt$ with $c>0$, $0\le v<1$, and $v+c\le 1$ (the latter condition ensures
that the wall nontrivially interacts with the particles), the asymptotic
distribution of the second class particle is a mixture of the uniform
distribution on the shortened segment $\left(-t,t(-1+2v+2\sqrt{c(1-v)})\right)$
and an atomic measure at the right end of this segment.

The paper is organized as follows. Section \ref{SectIntro} is the introduction.
In Section \ref{SectTightness} we use backwards paths and comparisons with
stationary TASEPs to prove weak process-level convergence of a tagged TASEP
particle to the Airy$_2$ process at the fluctuation scale. Section
\ref{SectRelationWall} explains the application of the color-position symmetry
to TASEPs with a wall. In Section \ref{SectScalingLimits} we prove our main
result by combining the results of the two previous sections. The final Section
\ref{SecSecondClass} contains a proof of the second class particle asymptotics.

\bigskip
\noindent\textbf{Acknowledgements.} A.~Borodin was partially supported by the
NSF grant DMS-1853981, and the Simons Investigator program. The work of P.L. Ferrari
was partly funded by the Deutsche Forschungsgemeinschaft (DFG, German Research Foundation) under Germany's Excellence Strategy - GZ 2047/1,
projekt-id 390685813 and by the Deutsche Forschungsgemeinschaft (DFG, German Research Foundation) - Projektnummer 211504053 - SFB 1060.

\newpage

\section{Tightness of the scaled particle process}\label{SectTightness}
In what follows, whenever we consider TASEP with different initial conditions, we always assume that they are coupled by the basic coupling~\cite{Li85b}. In other words, the evolution of the different processes occurs using the same jump trials. It is useful to have in mind the graphical construction of TASEP~\cite{Har65,Lig76}, see also Figure~\ref{FigCharacteristics} below. We describe TASEP not by the occupation variables of the sites, but rather by the positions of labelled particles. For all times $t$ and label $k\in \Z$, denote the position of the particle with label $k$ at time $t$ by $x_k(t)$. We use the convention $x_k(t)>x_{k+1}(t)$.

The goal of this section is to obtain the following results based on the particle representation (without having to make a detour through the last passage percolation model):
\begin{itemize}
\item In Proposition~\ref{propTightness} we show that for the step initial condition, that is $x_k(0)=-k+1$, $k\geq 1$, the rescaled particle process converges weakly to the Airy$_2$ process. As convergence of finite-dimensional distributions is known, we need a modulus of continuity estimate to get tightness.
\item In Theorem~\ref{thmComparison} we show that the increments of the particle process can be bounded by the ones of two Poisson processes, which originate from two TASEPs with stationary initial conditions. This result is central for getting the estimate on the modulus of continuity, and also for bounding the increments in a mesoscopic scale for the scaling limit of Section~\ref{SectScalingLimits}.
\end{itemize}
The comparison of TASEPs requires careful estimates, which will be done using the formalism of backwards paths discussed below. Backwards paths have been introduced in~\cite{Fer18}, see also~\cite{FN19}.

\subsection{The backwards path}
We define the following process running \emph{backwards} in time. First we define a process on the labels starting from time $t$ backwards to time $0$, denoted $N(t\downarrow \cdot)$, as follows:
\begin{itemize}
\item we set $N(t\downarrow t)=N$,
\item the jumps happen at times when a TASEP jump trial is suppressed: if at time\footnote{$\hat t^+$ means a time moment that is infinitesimally larger than $\hat t$.} $\hat t^+$ we have $N(t\downarrow \hat t^+)=n$ and at time $\hat t$ a jump of the TASEP particle $n$ is suppressed by the presence of particle $n-1$, then we define $N(t\downarrow \hat t)=n-1$.
\end{itemize}
The backwards path associated to the label $N$ at time $t$ is defined by setting
\begin{equation}
\pi_{N,t}=\{x_{N(t\downarrow u)}(u), u\in[0,t]\},
\end{equation}
see Figure~\ref{FigSpaceCut}.
\begin{figure}[t!]
\begin{center}
\psfrag{x}[lc]{$x$}
\psfrag{t}[lb]{$t$}
\psfrag{tau}[lb]{$\tau$}
\includegraphics[height=5cm]{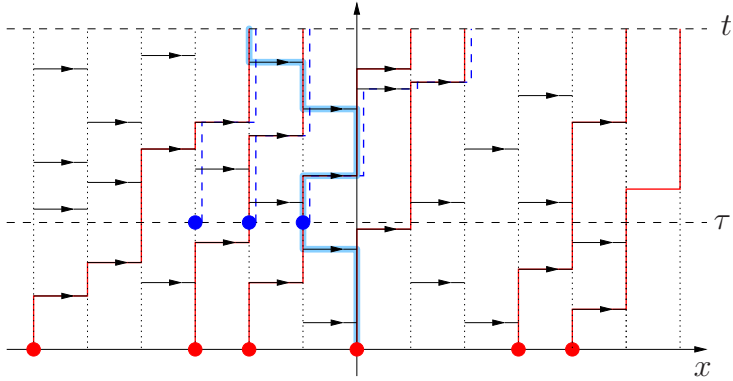}
\caption{The red solid lines are the trajectories of $\{x_n(s),0\leq s\leq t, \textrm{ all }n\}$. The thick light-blue line is the trajectory of $\{x_{N(s)}(s),0\leq s\leq t\}$. The blue dots and the dashed blue lines are the trajectories of the particles after setting step initial condition at time $\tau$.}
\label{FigSpaceCut}
\end{center}
\end{figure}

We denote by $x_n^{{\rm step},Z}(\tau,t)$ the particle process at time $t$
starting at time $\tau$ from the step initial condition with the right-most
particle at position $Z \in \Z$, i.e., with the initial condition
\begin{equation}
x_{n}^{{\rm step},Z}(\tau,t=\tau)=-n+Z+1, \quad n\geq 1.
\end{equation}
For $Z=0$ and $\tau=0$ we simply write $x_{n}^{\rm{step}}(t)$. Also, define $y_n^{Z}(\tau,t)=x_{n}^{{\rm step},Z}(\tau,t)-Z$.

In Proposition~3.4 of~\cite{Fer18} it is shown that
\begin{equation}\label{eq1}
x_N(t)=x_{N(t\downarrow\tau)}(\tau)+y^{x_{N(t\downarrow\tau)}(\tau)}_{N-N(t\downarrow\tau)+1}(\tau,t).
\end{equation}
Furthermore, for any other $n\leq N$ it holds\footnote{For TASEP, we thus have $x_N(t)=\inf_{n\leq N} \left\{ x_{n}(\tau)+y^{x_{n}(\tau)}_{N-n+1}(\tau,t)\right\}$, as previously proven by Sepp\"al\"ainen with a different approach in~\cite{Sep98c}.}
\begin{equation}\label{eq2}
x_N(t)\leq x_{n}(\tau)+y^{x_{n}(\tau)}_{N-n+1}(\tau,t).
\end{equation}
This can be also read in another way (the role of $n$ is taken by $N(t\downarrow\tau)$ and the role of $N$ is taken by $n$): for any $n\geq N(t\downarrow\tau)$,
\begin{equation}\label{eq3}
x_n(t)\leq x_{N(t\downarrow\tau)}(\tau)+y^{x_{N(t\downarrow\tau)}(\tau)}_{n-N(t\downarrow\tau)+1}(\tau,t).
\end{equation}

\begin{remark}
The important properties of the backwards path $\pi_{N,t}$ are \eqref{eq1}, \eqref{eq2}, \eqref{eq3}, and by construction it has almost surely only steps of size $\pm 1$.
\end{remark}

By construction of the backwards path, the position $x_N(t)$ is unchanged if particles strictly to the right of the path are moved to the right and particles strictly to the left are moved to the right but still keeping the exclusion constraint; the extreme situation is to create a step initial condition at position of the backwards path (which gives the formula \eqref{eq1}). If the backwards path ends weakly to the left of $0$, then we can move to the right particles initially strictly to the right of $0$ and get the initial configuration of a process we call $x^{{\rm right}}$, while if the backwards path ends strictly to the right of $0$, we can move to the right particles weakly to the left of the origin leading to a process we call $x^{{\rm left}}$. Therefore,
\begin{equation}\label{eqLeftRightDecomposition}
x_N(t)=\min\{x^{{\rm left}}_N(t),x^{{\rm right}}_N(t)\},
\end{equation}
for arbitrary $N$.

\subsection{Comparison inequalities}
For $t_1<t_2$, we would like to bound the increments of $x_N(t_2)-x_N(t_1)$ with
the increments of another process, namely $\tilde x_M(t_2)-\tilde x_M(t_1)$ for
suitable $M$ and $\tilde x$. The process $\tilde x$ will be taken to be a
stationary TASEP with some density $\rho$, because in this case for each fixed
$n$, $(\tilde x_n(t),t\geq 0)$ is a Poisson process with intensity $1-\rho$
(one-sided random walk with jumps to the right at rate $1-\rho$). This property
is coming from Burke's theorem~\cite{Bur56} and it was observed in~\cite{Spi70}
(Example~3.2), see also~\cite{FF96}.

\begin{prop}\label{propComparison1}
Let $x(t)$ and $\tilde x(t)$ be two TASEPs coupled by the basic coupling. Consider the tagged particle processes $x_N(t)$ and $\tilde x_M(t)$ on a time interval $[t_1,t_2]$. Assume that
\begin{equation}\label{eqchoiceM}
\tilde x_M(t_1)\leq x_N(t_1).
\end{equation}
Define the event
\begin{equation}\label{eqEvent}
{\cal E}_{t_1,M}^{t_2,N}=\{\exists \tau\leq t_1\, | \, x_{N(t_2\downarrow\tau)}(\tau)=\tilde x_{M(t_1\downarrow\tau)}(\tau)\}.
\end{equation}
If ${\cal E}_{t_1,M}^{t_2,N}$ takes place, then
\begin{equation}
x_N(t_2)-x_N(t_1)\geq \tilde x_M(t_2)-\tilde x_M(t_1).
\end{equation}
\end{prop}

\begin{proof}
Let us consider the backwards path $\pi_{N,t_2}$ associated with $x_N$ starting at time $t_2$ and the backwards path $\tilde \pi_{M,t_1}$ associated with $\tilde x_M$ starting at time $t_1$. Assume ${\cal E}_{t_1,M}^{t_2,N}$ is satisfied and set $x^*=\tilde x_{M(t_1\downarrow\tau)}(\tau)=x_{N(t_2\downarrow\tau)}(\tau)$. Then by \eqref{eq1} we have
\begin{equation}\label{eq1.9}
\begin{aligned}
x_N(t_2)&=x^*+y^{x^*}_{N-N(t_2\downarrow\tau)+1}(\tau,t_2),\\
\tilde x_M(t_1)&=x^*+y^{x^*}_{M-M(t_1\downarrow\tau)+1}(\tau,t_1),
\end{aligned}
\end{equation}
and by \eqref{eq2}
\begin{equation}\label{eq1.11b}
\begin{aligned}
x_N(t_1)&\leq x^*+y^{x^*}_{N-N(t_2\downarrow\tau)+1}(\tau,t_1),\\
\tilde x_M(t_2)&\leq x^*+y^{x^*}_{M-M(t_1\downarrow\tau)+1}(\tau,t_2).
\end{aligned}
\end{equation}
\begin{figure}[t!]
\begin{center}
\psfrag{xn}[cb]{$x_N$}
\psfrag{xm}[cb]{$\tilde x_M$}
\psfrag{t1}[lb]{$t_1$}
\psfrag{t2}[lb]{$t_2$}
\psfrag{tau}[lb]{$\tau$}
\psfrag{time}[lb]{time}
\psfrag{space}[rb]{space}
\psfrag{(a)}[ct]{(a)}
\psfrag{(b)}[ct]{(b)}
\includegraphics[height=7cm]{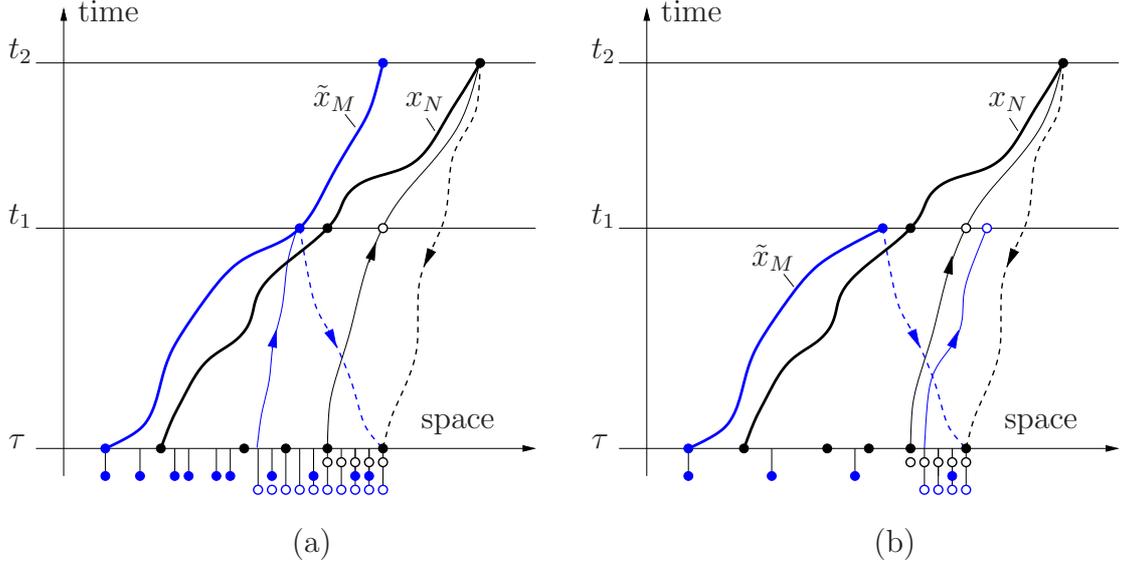}
\caption{The \emph{thick} solid lines are the evolution of particle $x_N$ and $\tilde x_M$, while the backwards path are dashed. The \emph{thin} solid lines are the evolution of particle $N$ (resp.\ $M$) after restarting with step initial condition at time $\tau$ and position $x^*$ for $x$ (resp.\ $\tilde x$). The solid dots are the particle configurations at time $\tau$ and the empty dots are the configurations after resetting to the step initial condition at time $\tau$. The picture (a) is for the case $N-N(t_2\downarrow \tau)+1\leq M-M(t_1\downarrow \tau)+1$, while (b) for the case $N-N(t_2\downarrow \tau)+1> M-M(t_1\downarrow \tau)+1$. The contradiction in (b) is that the blue empty dot at time $t_1$ should be at position $\tilde x_M(t_1)$, but it is also to the right of the black empty dot at time $t_1$.}
\label{FigComparison}
\end{center}
\end{figure}

Assume for a moment that
\begin{equation}\label{eq1.11}
N-N(t_2\downarrow\tau)+1\leq M-M(t_1\downarrow\tau)+1.
\end{equation}
Then, combining \eqref{eq1.9} and \eqref{eq1.11b} we get
\begin{equation}
\begin{aligned}
x_N(t_2)-x_N(t_1)&\geq  y^{x^*}_{N-N(t_2\downarrow\tau)+1}(\tau,t_2)-y^{x^*}_{N-N(t_2\downarrow\tau)+1}(\tau,t_1)\\
&\geq y^{x^*}_{M-M(t_1\downarrow\tau)+1}(\tau,t_2)-y^{x^*}_{M-M(t_1\downarrow\tau)+1}(\tau,t_1)\\
&\geq \tilde x_M(t_2)-\tilde x_M(t_1).
\end{aligned}
\end{equation}
The second inequality holds due to \eqref{eq1.11} and the fact that, for the step initial condition, particles starting to the left of particle $N$ can not move more than the distance traveled by particle $N$.

It remains to prove \eqref{eq1.11}. Assume that it is not true. Consider the evolution of the particles obtained by setting the step initial condition at time $\tau$ with the right-most particle at $x^*$, which we label as particle number $1$; i.e., we look at the process $\{x^*+y^{x^*}_{n}(\tau,t)\}_{n\geq 1}$. By assumption, particle $N_1=N-N(t_2\downarrow\tau)+1$ is to the left of particle $M_1=M-M(t_1\downarrow\tau)+1$. By \eqref{eq1.11b} particle $N_1$ at time $t_1$ is on $[x_N(t_1),\infty)$, and thus particle $M_1$ at time $t_1$ is strictly to the right of $x_N(t_1)$. In formulas,
\begin{equation}
x_N(t_1)\leq x^*+y^{x^*}_{N_1}(\tau,t_1)< x^*+y^{x^*}_{M_1}(\tau,t_1) = \tilde x_M(t_1),
\end{equation}
where the last equality comes from \eqref{eq1.9}. But by assumption \eqref{eqchoiceM} we have $\tilde x_M(t_1) \in (-\infty,x_N(t_1)]$, which is a contradiction.
\end{proof}

Interchanging the roles of $x$ and $\tilde x$ we get the following result.
\begin{prop}\label{propComparison2}
Let $x(t)$ and $\hat x(t)$ be two TASEPs under the basic coupling. Consider the tagged particles processes $x_N(t)$ and $\hat x_P(t)$ on the interval $[t_1,t_2]$. Assume that
\begin{equation}\label{eqchoiceP}
x_N(t_1)\leq \hat x_P(t_1).
\end{equation}
Define the event
\begin{equation}\label{eqEventB}
{\cal \hat E}_{t_1,N}^{t_2,P}=\{\exists \tau\leq t_1\, | \, x_{N(t_1\downarrow\tau)}(\tau)=\hat x_{P(t_2\downarrow\tau)}(\tau)\}.
\end{equation}
If ${\cal \hat E}_{t_1,N}^{t_2,P}$ is satisfied, then
\begin{equation}
x_N(t_2)-x_N(t_1)\leq \hat x_P(t_2)-\hat x_P(t_1).
\end{equation}
\end{prop}

\subsection{End-point localization of backwards paths}
Here with prove two lemmas which imply that with high probability the events ${\cal E}$ and ${\cal \hat E}$ do occur from \eqref{eqEvent} and \eqref{eqEventB}.

\begin{lem}\label{lemLocalStep}
Let $x(t)$ be the TASEP with the step initial condition. Let $N=\alpha T$ with $\alpha\in(0,1)$ and $t=T-\varkappa T^{2/3}$ with $\varkappa$ in a bounded set. Define $\alpha_T=\alpha T/t$. Then there exist constants $C,c>0$ independent of $N$ such that for all $K_1,K_2>0$ we have
\begin{equation}\label{eq1.18}
\Pb(|x_{N(t\downarrow 0)}(0)|\geq K_1 T^{1/3})\leq C e^{-c K_1}
\end{equation}
and
\begin{equation}\label{eq1.16}
\Pb\left(|x_{N}(t)-(1-2\sqrt{\alpha_T})t|\geq K_2 T^{1/3}\right)\leq C e^{-c K_2},
\end{equation}
uniformly for all $N$ large enough.
\end{lem}
\begin{proof}
For the step initial condition $x_{N(t\downarrow 0)}(0)=-N(t\downarrow 0)+1\leq 0$. Thus, by \eqref{eq1} at time $\tau=0$ we have
\begin{equation}\label{eq1.20}
x_N(t)=x^{{\rm step},-N(t\downarrow 0)+1}_{N-N(t\downarrow 0)+1}(t),
\end{equation}
that is, $x_N(t)$ can be obtained from the step initial condition where the particles on $[-N(t\downarrow 0)+2,0]$ are removed. Let us estimate the probability that $N(t\downarrow 0)\geq K_1 T^{1/3}+1$.

We have\footnote{In the r.h.s.\ of \eqref{eq1.21} we have used a strict inequality. The reason is that the event \mbox{$\{x_N(t)=\min_{n\geq K_1 T^{1/3}+1}x^{{\rm step},-n+1}_{N-n+1}(t))\}$} does not imply $N(t\downarrow 0)< K_1 T^{1/3}+1$, because several values of $n$ could minimize $x^{{\rm step},-n+1}_{N-n+1}(t))$.} (using \eqref{eq1.20})
\begin{equation}\label{eq1.21}
\begin{aligned}
\Pb(N(t\downarrow 0)< K_1 T^{1/3}+1)&\geq \Pb(x_N(t)<\min_{n\geq K_1 T^{1/3}+1}x^{{\rm step},-n+1}_{N-n+1}(t))\\
&= \Pb(x_N(t)<x^{{\rm step},-K_1 T^{1/3}}_{N-K_1 T^{1/3}}(t)),
\end{aligned}
\end{equation}
where we used that $x^{{\rm step},-n+1}_{N-n+1}(t)$ is weakly increasing in $n$. Next, for any constant $A$,
\begin{equation}
\begin{aligned}
\Pb(x_N(t)<x^{{\rm step},-K_1 T^{1/3}}_{N-K_1 T^{1/3}}(t))&\geq \Pb(x_N(t)\leq A <x^{{\rm step},-K_1 T^{1/3}}_{N-K_1 T^{1/3}}(t))\\
&\geq 1-\Pb(\{x_N(t)> A\}\cup\{A\geq x^{{\rm step},-K_1 T^{1/3}}_{N-K_1 T^{1/3}}(t)\})\\
&\geq 1-\Pb(x_N(t)> A)-\Pb(A\geq x^{{\rm step},-K_1 T^{1/3}}_{N-K_1 T^{1/3}}(t)).
\end{aligned}
\end{equation}
Thus
\begin{equation}\label{eq1.23}
\Pb(N(t\downarrow 0)\geq K_1 T^{1/3})\leq \Pb(x_N(t)> A)+\Pb(A\geq x^{{\rm step},-K_1 T^{1/3}}_{N-K_1 T^{1/3}}(t)).
\end{equation}
We choose
\begin{equation}
A=(1-2\sqrt{\alpha_T})t+\tfrac12 K_1 (\alpha^{-1/2}-1)T^{1/3}.
\end{equation}

To bound the first term, with $\sigma=\frac{(1-\sqrt{\alpha})^{2/3}}{\alpha^{1/6}}$ we have (see Lemma~\ref{LemOnePointBounds})
\begin{equation}
\lim_{T\to\infty}\Pb\left(x_{\alpha T}(t)\geq A\right)=\lim_{T\to\infty}\Pb\left(x_{\alpha T}(t)\geq (1-2\sqrt{\alpha_T})t-s \sigma T^{1/3}\right) = F_{\rm GUE}(s)
\end{equation}
with $s=-\tfrac12 K_1 (\alpha^{-1/2}-1)/\sigma$. The tails for large but finite time $T$ are also known to be (super)-exponential. In particular, for all $T$ large enough and $s_1<0$, see \eqref{eqOnePointLowerTail},
\begin{equation}\label{eq1.26}
\Pb\left(x_{\alpha T}(t)\geq (1-2\sqrt{\alpha_T})t-s_1 \sigma T^{1/3}\right) \leq C e^{-c |s_1|^{3/2}}.
\end{equation}
\eqref{eq1.16} follows from this bound, together with the bound of the upper tail \eqref{eqOnePointUpperTail}.
By our choice of $A$, $\Pb\left(x_{\alpha T}(t)\geq A\right)\leq C e^{-c K_1^{3/2}}$ for other constants $C,c>0$.

To bound the second term in \eqref{eq1.23}, let $\hat \alpha_T=\frac{\alpha T-K_1 T^{1/3}}{T-\tau T^{2/3}}$. Then
\begin{equation}
\begin{aligned}
&\Pb\left(x^{{\rm step},-K_1 T^{1/3}}_{\alpha T-K_1 T^{1/3}}(t)> (1-2\sqrt{\hat\alpha_T})t-K_1 T^{1/3}-s_2 \sigma T^{1/3}\right) \\
&= \Pb\left(x^{{\rm step}}_{\alpha T-K_1 T^{1/3}}(t)> (1-2\sqrt{\hat\alpha_T})t-s_2 \sigma T^{1/3}\right)\stackrel{T\to\infty}{\longrightarrow} F_{\rm GUE}(s_2),
\end{aligned}
\end{equation}
and, for all $T$ large enough, see \eqref{eqOnePointUpperTail},
\begin{equation}
\Pb\left(x^{{\rm step},-K_1 T^{1/3}}_{\alpha T-K_1 T^{1/3}}(t)\leq (1-2\sqrt{\hat\alpha_T})t-K_1 T^{1/3}-s_2 \sigma T^{1/3}\right) \leq C e^{-c s_2}.
\end{equation}
A computation gives, as $T\to\infty$,
\begin{equation}
(1-2\sqrt{\hat\alpha_T})t-K_1 T^{1/3} = (1-2\sqrt{\alpha_T})t +K_1(\alpha^{-1/2}-1)T^{1/3}+O(1).
\end{equation}
This corresponds to setting $s_2=\tfrac12K_1(\alpha^{-1/2}-1)/\sigma+o(1)$, which implies
\begin{equation}
\Pb\left(x^{{\rm step},-K_1 T^{1/3}}_{\alpha T-K_1 T^{1/3}}(t)\leq A\right)\leq C e^{-c K_1}
\end{equation}
for some other constants $C,c>0$.

Since $N(t\downarrow 0)\geq 1$, \eqref{eq1.18} follows from the bounds on the two terms of \eqref{eq1.23}.
\end{proof}

\begin{lem}\label{lemLocalStat}
Consider the TASEP with the stationary initial condition with density $\rho$. Then there exist constants $C,c>0$ independent of $t$ such that for all $K>0$ and any $N\in\Z$
\begin{equation}
\Pb(|x^\rho_N(t)-x^\rho_{N(t\downarrow 0)}(0)-(1-2\rho) t|\geq K t^{2/3})\leq C e^{-c K},
\end{equation}
uniformly for all $t$ large enough.
\end{lem}
\begin{proof}
Consider a density $\rho_+=\rho_0+\tilde \kappa t^{-1/3}$ with $\tilde\kappa>0$ and set $\chi_+=\rho_+(1-\rho_+)$, $\chi_0=\rho_0(1-\rho_0)$. In the stationary TASEP with density $\rho_+$, denoted as $x^{\rho_+}$, choose the particle number $M=\rho_+^2 t-2\tilde \kappa \rho_+ t^{2/3}$ (where particle with label $0$ is the first one strictly to the right of the origin). Then, see Theorem~1.6 of~\cite{BFP09},
\begin{equation}
\begin{aligned}
&\lim_{t\to\infty} \Pb\left(x^{\rho_+}_M(t)\geq (1-2\rho_+)t+2\tilde \kappa t^{2/3}-(1-\rho_+)\chi_+^{-1/3} s t^{1/3}\right)\\
&=\lim_{t\to\infty} \Pb\left(x^{\rho_+}_M(t)\geq (1-2\rho_0)t-(1-\rho_0)\chi_0^{-1/3} s t^{1/3}\right) = F_{{\rm BR},\tilde \kappa/\chi_0^{1/3}}(s),
\end{aligned}
\end{equation}
where $F_{{\rm BR},w}$ is the Baik-Rains distribution function with parameter $w$. In addition, for all large $t$ (see Lemma~\ref{LemOnePointStat})
\begin{equation}\label{eq1.29}
\Pb\left(|x^{\rho_+}_M(t) -(1-2\rho_0)t|\geq K t^{1/3}\right)\leq C e^{-c K}
\end{equation}
for some constants $C,c>0$.

Now consider TASEPs with two initial conditions modified as follows: (a) $x^{\rho_+,{\rm left}}$ consists of $x^{\rho_+}$ in which all particles weakly to the left of $0$ are pushed to the right until $0$ to create the step initial condition on $\{\ldots,-2,-1,0\}$ and particles strictly to the right of $0$ are unchanged, (b) $x^{\rho_+,{\rm right}}$ consists of removing all particles starting strictly to the right of $0$ of $x^{\rho_+}$, keeping the particles to the left of $0$ unchanged (the numbering of the remaining particles is unchanged).

By \eqref{eqLeftRightDecomposition} and the discussion above it, we have
\begin{equation}\label{eq1.31}
\{x^{\rho_+}_{M(t\downarrow 0)}(0)>0\}=\{M(t\downarrow 0)\leq 0\} \supseteq \{x^{\rho_+,{\rm right}}_M(t) > x^{\rho_+,{\rm left}}_M(t)\}.
\end{equation}
The inclusion is not necessarily an equality since it could also happen that the two random variables in the last expression are equal. Thus we have, for any choice of $A$,
\begin{equation}
\begin{aligned}
\Pb(x^{\rho_+}_{M(t\downarrow 0)}(0)>0)&\geq \Pb(x^{\rho_+,{\rm right}}_M(t) > A \geq x^{\rho_+,{\rm left}}_M(t))\\
&\geq 1-\Pb(x^{\rho_+,{\rm right}}_M(t) \leq A) - \Pb(x^{\rho_+,{\rm left}}_M(t)> A),
\end{aligned}
\end{equation}
which gives
\begin{equation}\label{eq1.32}
\Pb(x^{\rho_+}_{M(t\downarrow 0)}(0)\leq 0) \leq \Pb(x^{\rho_+,{\rm right}}_M(t) \leq  A) +\Pb(x^{\rho_+,{\rm left}}_M(t)> A).
\end{equation}
We choose $A=(1-2\rho_0)t+\tilde \kappa^2 t^{1/3}$, since typically $x^{\rho_+,{\rm left}}_M(t)$ is to the left of $A$ and $x^{\rho_+,{\rm right}}_M(t)$ is to the right of $A$.

The two initial condition used here can be analyzed directly using determinantal formulas, but the bounds for precisely this case have not been written down before. To avoid redoing standard asymptotic analysis, we opt to use the connection to the last passage percolation (LPP) where such computations have been made. We want to emphasize that it is not necessary.

We have
\begin{equation}
\begin{aligned}
&\Pb(x^{\rho_+,{\rm left}}_M(t) >  A) = \Pb(L^{\rho_+}_{-} (M+A,M)\leq t),\\
&\Pb(x^{\rho_+,{\rm right}}_M(t) \leq A) = \Pb(L^{\rho_+}_\vert (M+A,M) > t).
\end{aligned}
\end{equation}
Here $L^{\rho_+}_\vert$ is the LPP with ${\rm exp}(1)$ random variables on $\{(i,j)| i\geq 1,j\geq 1\}$ and ${\rm exp}(\rho_+)$ on $\{(0,j)| j\geq 1\}$, while $L^{\rho_+}_{-}$ is the LPP with the same bulk randomness and with ${\rm exp}(1-\rho_+)$ on $\{(i,0)|i\geq 1\}$.

In Lemma~2.5 of~\cite{FO17}, there are bounds on exit point probabilities bounding precisely $\Pb(L^{\rho_+}_\vert (M+A,M)> t)$ and $\Pb(L^{\rho_+}_{-} (M+A,M) \leq t)$; see (3.7) of~\cite{FO17}. The variables in~\cite{FO17} should be matched as follows:
\begin{equation}
x \to t,\quad n\to M,\quad \gamma^2 n \to M+A,\quad \kappa \to \tilde\kappa (n/t)^{1/3}=\tilde\kappa\rho_0^{2/3}+O(t^{-2/3}).
\end{equation}
Then the bounds used for the proof of Lemma~2.5, see Lemma~3.3 of~\cite{FO17}, lead to
\begin{equation}\label{eq1.35}
\Pb(L^{\rho_+}_\vert (M+A,M)> t)\leq C e^{-c \tilde\kappa^2},\quad \Pb(L^{\rho_+}_{-} (M+A,M) \leq t)\leq C e^{-c\tilde \kappa^3}
\end{equation}
for some constants $C,c>0$ (those can be taken uniformly for $\rho_+$ in a compact subset of $(0,1)$).

Now, we set $\rho_0=\rho_+-\kappa t^{-1/3}$, so that $\rho_+=\rho$. By \eqref{eq1.32}-\eqref{eq1.35} we have, with probability at least $1-2C e^{-c\tilde\kappa^2}$, that $x^{\rho}_M(t)-x^{\rho}_{M(t\downarrow 0)}(t)\leq x^\rho_M(t)$. By \eqref{eq1.29} we further have that with probability at least $1-C e^{-c K}$, $x^\rho_M(t)\leq (1-2\rho_0)t-K t^{2/3}=(1-2\rho)t+(2\tilde\kappa-K) t^{2/3}$. Thus, choosing $\tilde\kappa=K$ we get that
\begin{equation}\label{eq1.36}
\Pb(x^{\rho}_M(t)-x^{\rho}_{M(t\downarrow 0)}(0)\geq (1-2\rho)t+K t^{2/3})\leq C e^{-c K}
\end{equation}
for some new constants $C,c>0$. With similar arguments one shows that
\begin{equation}\label{eq1.37}
\Pb(x^{\rho}_M(t)-x^{\rho}_{M(t\downarrow 0)}(0)\leq (1-2\rho)t-K t^{2/3})\leq C e^{-c K}.
\end{equation}
By translation invariance, the statements \eqref{eq1.36} and \eqref{eq1.37} hold for any $M$, which is our result.
\end{proof}

\subsection{Comparison inequality and tightness}
First let us see that it is enough to verify the events ${\cal E}$ and $\cal \hat E$ at the largest and smallest times of a given interval.
\begin{lem}\label{lemGlobalEvents}
Consider two times $t_1,t_2$ such that $t\leq t_1<t_2\leq T$. Then, with notations \eqref{eqEvent}, \eqref{eqEventB}, we have
\begin{equation}\label{eqEinclusions}
{\cal E}_{t,M}^{T,N} \subseteq {\cal E}_{t_1,M}^{t_2,N},\quad \textrm{and}\quad {\cal \hat E}_{t,M}^{T,N} \subseteq  {\cal \hat E}_{t_1,M}^{t_2,N}.
\end{equation}
\end{lem}
\begin{proof}
Assume that ${\cal E}_{t,M}^{T,N}$ occurs. Then we have $\tilde x_M(t)\leq x_N(t)$ (by assumption), and hence (via basic coupling) $\tilde x_M(T)\leq x_N(T)$. For any $u \in [t,T]$, we have that the backwards path from $x_N(u)$ is to the left of the backwards path from $x_N(T)$. This implies that the backwards path from $\tilde x_M(t)$ will intersect the one from $x_N(u)$. Thus ${\cal E}_{t,M}^{T,N}$ implies ${\cal E}_{t,M}^{u,N}$ for all $u\in[t,T]$. In turn, since $\tilde x_M(t)\leq x_N(t)$ and the event ${\cal E}_{t,M}^{u,N}$ implies, by Proposition~\ref{propComparison1}, that the increment $x_N(u)-x_N(t)$ is larger than $\tilde x_M(u)-\tilde x_M(t)$, we also have $\tilde x_M(u)\leq x_N(u)$ for all $u\in [t,T]$. But this then implies that the backwards path from $\tilde x_M(t_1)$ intersects the backwards path from $x_N(t_2)$ for all $t_1<t_2$ in $[t,T]$.

For $\cal \hat E$ it is similar. See also Figure~\ref{FigSandwitching} for an illustration.
\end{proof}
\begin{figure}[t!]
\begin{center}
\psfrag{time}[lc]{time}
\psfrag{space}[lc]{space}
\psfrag{xn}[rc]{$x_N$}
\psfrag{xm}[rc]{$\tilde x_M$}
\psfrag{t}[lc]{$t$}
\psfrag{t1}[lc]{$t_1$}
\psfrag{t2}[lc]{$t_2$}
\psfrag{T}[l<c]{$T$}
\psfrag{EtT}[lc]{${\cal E}_{t,M}^{T,N}$}
\psfrag{Et1t2}[lc]{${\cal E}_{t_1,M}^{t_2,N}$}
\includegraphics[height=5cm]{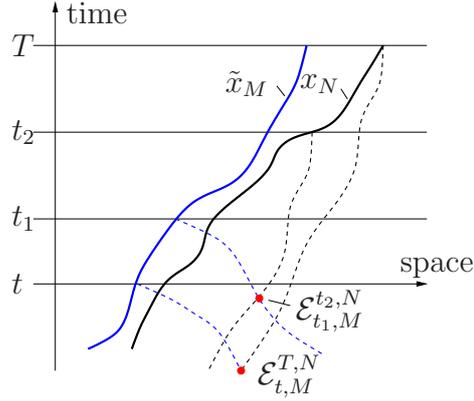}
\caption{The solid lines are the evolution of particle $x_N$ and $\tilde x_M$, while the dashed paths are backwards path starting at different times. The red dots represents the occurrence of the events ${\cal E}_{t,M}^{T,N}$ and ${\cal E}_{t_1,M}^{t_2,N}$. One sees that ${\cal E}_{t,M}^{T,N}$ implies ${\cal E}_{t_1,M}^{t_2,N}$.}
\label{FigSandwitching}
\end{center}
\end{figure}

With the above results, we immediately have an estimate on the probabilities of ${\cal E}_{t_1,M}^{t_2,N}$ and ${\cal \hat E}_{t_1,N}^{t_2,P}$.
\begin{prop}\label{propComparison}
Let $x(t)$ be the TASEP with the step initial condition. Set \mbox{$N=\alpha T$}, $\alpha\in (-1,1)$, and $t=T-\varkappa T^{2/3}$ with $\varkappa>0$ fixed. Let $\rho_0=\sqrt{\alpha_T}=\sqrt{\alpha T/t}$ be the average density of particles around $x_N(t)$. Define
\begin{equation}
\rho_\pm = \rho_0\pm \kappa t^{-1/3}=\sqrt{\alpha}+\left(\tfrac12\varkappa \sqrt{\alpha}\pm \kappa \right)t^{-1/3}+\Or(t^{-2/3})
\end{equation}
with $\kappa>0$. Consider two stationary processes: $\tilde x(t)=x^{\rho_+}(t)$ having density $\rho_+$, and $\hat x(t)=x^{\rho_-}(t)$  having density $\rho_-$. Define the index $M$ to be the smallest one such that $x^{\rho_+}_M(t)\leq  x_N(t)$ and the index $P$ to be the largest one such that $x_N(t) \leq  x^{\rho_-}_P(t)$, and recall the events ${\cal E}_{t,M}^{T,N}$ and ${\cal\hat E}_{t,N}^{T,P}$ of Propositions~\ref{propComparison1} and~\ref{propComparison2}. Then
\begin{equation}
\Pb({\cal E}_{t_1,M}^{t_2,N}\cap {\cal\hat E}_{t_1,N}^{t_2,P}\textrm{ for all }t\leq t_1<t_2\leq T)\geq 1-C e^{-c\kappa},
\end{equation}
for some constants $C,c>0$, where the constants are uniform for all $T$ large enough.
\end{prop}
\begin{proof}
By Lemma~\ref{lemGlobalEvents} it is enough to bound the probabilities of ${\cal E}_{t,M}^{T,N}$  and ${\cal \hat E}_{t,N}^{T,P}$. Lemma~\ref{lemLocalStep} with $\varkappa=0$ gives the localization of $x_{N(T\downarrow 0)}(0)$ in a $T^{1/3}$ window and Lemma~\ref{lemLocalStep} with $\varkappa$ gives the localization of $x_{N}(t)-(1-2\rho_0)t$ in a $T^{1/3}$ window. Taking for instance $K_1=K_2=\kappa/2$ we have that
\begin{equation}
\Pb(x_N(t)-x_{N(T\downarrow 0)}(0)< (1-2\rho_0)t-\kappa T^{1/3})\leq 2C e^{-c\kappa/2}.
\end{equation}
Lemma~\ref{lemLocalStat} with $t=T-\varkappa T^{2/3}$ and $K=\kappa$ gives
\begin{equation}
\begin{aligned}
&\Pb(x^{\rho_+}_M(t)-x^{\rho_+}_{M(t\downarrow 0)}(0)>(1-2\rho_+)t+\kappa t^{2/3})\\
= &\Pb(x^{\rho_+}_M(t)-x^{\rho_+}_{M(t\downarrow 0)}(0)>(1-2\rho_0)t-\kappa t^{2/3})\leq C e^{-c\kappa}.
\end{aligned}
\end{equation}
Since $x^{\rho_+}_M(t)-x_N(t)$ is $\Or(1)$, this implies that uniformly for all $T$ large enough, $x^{\rho_+}_{M(t\downarrow 0)}(0)>x_{N(T\downarrow 0)}(0)$ with probability at least $1-C e^{-c\kappa}$ (for some new constants $C,c>0$) and thus the event ${\cal E}_{t,M}^{T,N}$ occurs. Similarly one gets a bound for ${\cal \hat E}_{t,N}^{P,T}$.
\end{proof}

Proposition\ref{propComparison} implies the following comparison inequalities.
\begin{thm}\label{thmComparison}
Let us consider the setting of Proposition~\ref{propComparison}. Let $T-\varkappa T^{2/3}\leq t_1=T-\tau_1 T^{2/3}<t_2=T-\tau_2 T^{2/3}\leq T$. Then with probability at least $1-C e^{-c\kappa}$,
\begin{equation}
x^{\rho_+}_M(t_2)-x^{\rho_+}_M(t_1)\leq x_N(t_2)-x_N(t_1)\leq x^{\rho_-}_P(t_2)-x^{\rho_-}_P(t_1).
\end{equation}
\end{thm}
\begin{proof}
This follows directly from Propositions~\ref{propComparison1},~\ref{propComparison2}, and~\ref{propComparison}.
\end{proof}

Define the constants
\begin{equation}
c_1=\frac{(1-\sqrt{\alpha})^{2/3}}{\alpha^{1/6}}\quad\textrm{and} \quad c_2=\frac{2(1-\sqrt{\alpha})^{1/3}}{\alpha^{1/3}}.
\end{equation}
The law of large numbers approximation of $x_{\alpha T}(T-c_2 \tau T^{2/3})$ is given by
\begin{equation}
\begin{aligned}
x_{\alpha T}(T-c_2\tau T^{2/3})&\simeq (1-2\sqrt{\alpha T/(T-c_2\tau T^{2/3})})(T-c_2\tau T^{2/3})\\
&\simeq \left(1-2 \sqrt{\alpha }\right) T+c_2\tau \left(\sqrt{\alpha } -1\right)T^{2/3} +\frac{1}{4} \sqrt{\alpha } c_2^2\tau^2 T^{1/3}.
\end{aligned}
\end{equation}

\begin{prop}\label{propTightness}
Let us define the rescaled process
\begin{equation}
X_T(\tau):=\frac{x_{\alpha T}(T-c_2\tau T^{2/3})-\mu(\tau,T)}{-c_1 T^{1/3}}.
\end{equation}
where $\mu(\tau,T)=\left(1-2 \sqrt{\alpha }\right) T-c_2\tau \left(1-\sqrt{\alpha}\right)T^{2/3}$.
Then we have
\begin{equation}\label{eq1.51}
\lim_{T\to\infty} X_T(\tau) = {\cal A}_2(\tau)-\tau^2,
\end{equation}
where ${\cal A}_2$ is the Airy$_2$ process. The convergence is the weak convergence on the space of continuous functions on compact intervals.
\end{prop}
\begin{proof}
The convergence of \eqref{eq1.51} in the sense of finite-dimensional distribution is known from the comparison with the LPP, see~\cite{BP07} (for a discrete time setting, see Theorem~2-2 of~\cite{SI07}). It can be obtained also directly using the determinantal formula for the joint distributions of finitely many particles (it is a particular case of (2.23) in~\cite{BF07}). Through the connection to LPP one can get tightness~\cite{FO17} of the scaled tagged particle process $X_T$. Here we derive the tightness property directly in the particle representations without employing the LPP correspondence.

Define the modulus of continuity for the rescaled process $X_T$ by
\begin{equation}
\omega_T(\delta)=\sup_{\stackrel{0\leq \tau_1,\tau_2\leq \varkappa:}{|\tau_2-\tau_1|\leq \delta}}|X_T(\tau_1)-X_T(\tau_2)|.
\end{equation}
Since $X_T(0)$ is a tight random variable, to show tightness of the process (in the space of continuous functions on bounded intervals) we still need to control the modulus of continuity: we need to prove that for any $\e,\tilde\e>0$, there exists a $\delta>0$ and a $T_0$ such that (see Theorem 8.2 in~\cite{Bil68})
\begin{equation}\label{eq1.46}
\Pb(\omega_T(\delta)\geq \e)\leq \tilde\e,
\end{equation}
for all $T\geq T_0$. The tightness, together with convergence of finite-dimensional distributions, will imply the desired weak convergence.

Let us define the rescaled processes with densities $\rho_+$ and $\rho_-$ (defined as in Proposition~\ref{propComparison} with $t=T-\varkappa T^{2/3}$)
\begin{equation}
B^\pm_T(\tau):=\frac{x^{\rho_\pm}_{M}(T-c_2\tau T^{2/3})-\mu(\tau,T)}{-c_1 T^{1/3}}.
\end{equation}
Since $x^{\rho_+}_M(t)$ (resp.\ $x^{\rho_-}_P(t)$) is a Poisson process with parameter $1-\rho_+$ (resp.\ $1-\rho_-$), we immediately have the weak convergence to the Brownian motion (see Lemma~\ref{LemPPtoBM}), namely
\begin{equation}\label{eqBMcvg}
\lim_{T\to\infty} B^\pm_T(\tau)-B^{\pm}_T(0) = -\tau v_\pm +\sqrt{2}{\cal B}^\pm(\tau),
\end{equation}
where ${\cal B}^+$ and ${\cal B}^-$ are standard Brownian motions, and $v_\pm=(\pm\kappa + \tfrac12\sqrt{\alpha}\varkappa)c_2/c_1$ are the drifts.

For any $\e>0$ and $T$ large enough, by Proposition~\ref{propComparison},
\begin{equation}
\Pb(\omega_T(\delta)\geq \e)\leq C e^{-c\kappa} + \Pb\left(\{\omega_T(\delta)\geq \e\}\cap {\cal E}_{t,M}^{T,N}\cap {\cal\hat E}_{t,N}^{T,P}\right).
\end{equation}
On ${\cal E}_{t,M}^{T,N}\cap {\cal\hat E}_{t,N}^{T,P}$ we have the inequality
\begin{equation}
|X_T(\tau_2)-X_T(\tau_1)|\leq |B_T^+(\tau_2)-B_T^+(\tau_1)|+|B_T^-(\tau_2)-B_T^-(\tau_1)|.
\end{equation}
Choose $\delta>0$ such that the drift correction is smaller than $\e/2$; that is $|v_\pm|\delta\leq \e/2$. Hence, it is enough to bound
\begin{equation}\label{eq1.50}
\Pb\Bigg(\sup_{\stackrel{0\leq \tau_1,\tau_2\leq \varkappa:}{|\tau_2-\tau_1|\leq \delta}} |B_T^+(\tau_2)-B_T^+(\tau_1)+v_+ (\tau_2-\tau_1)|\geq \e/4\Bigg)
\end{equation}
and the same for $B_T^+$ replaced by $B_T^-$. Dividing the interval $[0,\varkappa]$ in pieces of length $\delta$, we have
\begin{equation}\label{eq1.52}
\eqref{eq1.50}\leq \frac{\varkappa}{\delta}\Pb\left(\sup_{0\leq \tau \leq \delta} |B_T^+(\tau)-B_T^+(0)+v_+ \tau|\geq \e/12\right)\leq \frac{\varkappa}{\delta}e^{-c\e^2/\delta},
\end{equation}
for some constant $c>0$, uniformly for all $T$ large enough. The bound is obtained by a standard computation with Doob's maximal inequality followed by the exponential Chebyshev inequality. The reason is that $B_T^+(\tau)-B_T^+(0)+v_+ \tau$ converges weakly to a Brownian motion by Donsker's theorem. A similar bound is obtained for $B_T^-$ instead of $B_T^+$.

Thus, for a fixed $\tilde\e$, take $\kappa>0$ large enough such that $C e^{-c\kappa}\leq \tilde\e/2$. Then take $\delta$ small enough so that $\eqref{eq1.52}\leq \tilde\e/2$. This implies \eqref{eq1.46} and thus we have tightness, which was the missing ingredient off the weak convergence.
\end{proof}

\section{Relation between TASEPs with a wall and without constraints}\label{SectRelationWall}
In the following proposition we show that, for the step initial condition, the one-point distribution for a TASEP with a boundary wall moving to the right deterministically can be related with the distribution of another TASEP without constraints.
\begin{prop}\label{PropStartingIdentity}
Let $f$ be a non-decreasing function on $\R_{\geq 0}$ with $f(0) = 0$. Let $x^f(t)$ be a TASEP starting with the step initial condition, i.e., $x^f_k(0)=-k+1$, $k\geq 1$, but with its first particle blocked by a wall which has position $f(t)$ at time $t$, i.e., $x^f_1(t)\leq f(t)$ for any time $t\geq 0$ \footnote{In other words, the jumps of the first particle that would violate the constraint $x^f_1(t) \leq f(t)$ are suppressed.}. Then for any $n \ge 1$ we have
\begin{equation}\label{eqStartingPoint}
\Pb(x^f_n(T) > s)=\Pb(x_n(t)> s-f(T-t)\textrm{ for all }t\in [0,T]),
\end{equation}
where $x(t)$ is a TASEP with the step initial condition evolving without the wall constraint.
\end{prop}
The proof of Proposition~\ref{PropStartingIdentity} is given in Section~\ref{proofStartingIdentity} below.

\begin{remark}
In Proposition~\ref{proofStartingIdentity}, we can replace the condition $f(0)=0$ with $f(0)\geq 0$, in which case the identity \eqref{eqStartingPoint} has to be replaced by
\begin{equation}
\Pb(x^f_n(T) > s)=\Pb(x_n(t)> s-f(T-t)\textrm{ for all }t\in [0,T], x_n(T)>s).
\end{equation}
\end{remark}

\subsection{Updates of multi-species TASEPs}
We start with a description of a colored (or multi-species, or multi-type) version of TASEP. We consider an interacting particle system in which particles live on the integer lattice $\Z$ and each integer location contains exactly one particle. The set of colors is taken to be $\Z \cup \{ +\infty\}$.

A particle configuration is a map $\eta: \Z \to \Z \cup \{ +\infty \}$, where we call $\eta(z)$ the \emph{color} of a particle at $z\in\Z$. When $\eta(z)=+\infty$ we will think of $z$ as being empty. Let $\mathfrak C$ be the set of all configurations. For a transposition $(z,z+1)$ with $z \in \Z$, let $\sigma_{(z,z+1)}: \mathfrak C \to \mathfrak C$, be a swap operator defined by
\begin{equation}
(\sigma_{(z,z+1)} \eta) (i) =
\begin{cases}
\eta(z+1), \qquad & i=z, \\
\eta(z), \qquad & i=z+1, \\
\eta(i), \qquad & i \in \Z \backslash \{ z,z+1 \}.
\end{cases}
\end{equation}
Define a \textit{totally asymmetric swap} operator at $z\in \Z$, $W_{(z,z+1)}: \mathfrak C \to \mathfrak C$, via
\begin{equation}
(W_{(z,z+1)} \eta) =
\begin{cases}
\eta, \qquad & \mbox{if $\eta(z) \ge \eta(z+1)$}, \\
\sigma_{(z,z+1)} \eta, \qquad & \mbox{if $\eta(z) < \eta(z+1)$}.
\end{cases}
\end{equation}
In words, the operator $\sigma_{(z,z+1)}$ exchanges the colors of particles at $z, z+1$, and $W_{(z,z+1)}$ realizes that exchange only if the particle that is initially on the left (at $z$) has a smaller color.

Any bijection of integers $s: \Z \to \Z$ can be viewed as a particle configuration by setting $\eta(z)=s(z)$. Such particle configurations will be especially important for us because of the following result.

\begin{prop}\label{prop:sym-basic}
Let $\mathrm{id} : \Z \to \Z$ be the identity bijection. Then, for any $k \in \Z$ and for any integers $z_1, \dots, z_k$ one has
\begin{equation}
W_{(z_k,z_k+1)} \dots W_{(z_2,z_2+1)} W_{(z_1,z_1+1)} \id= \mathrm{inv} \left( W_{(z_1,z_1+1)} W_{(z_2,z_2+1)} \dots W_{(z_k,z_k+1)} \id \right),
\end{equation}
where in the right-hand side $\mathrm{inv}$ denotes the map of taking the inverse of a permutation.
\end{prop}
In a probabilistic setting, Proposition~\ref{prop:sym-basic} was proved in~\cite[Lemma 2.1]{AHR09}, see~\cite{AAV11} and~\cite{BB19} for generalizations. In an equivalent algebraic setting, it turns out to be a well-known involution in the Hecke algebra, see~\cite{B20},~\cite{G20}.

\subsection{Two multi-species TASEPs}\label{proofStartingIdentity}
Now let us define a continuous-time multi-species TASEP. Consider a collection of jointly independent Poisson processes $\{ \mathcal P (z) \}_{z \in \Z}$, where $\mathcal P (z)$ has a state space $\R_{\ge 0}$ and rate $1$. Let $\eta_0 \in \mathfrak C$ be a (either deterministic or random) particle configuration which plays the role of an initial condition. Of course, in the random case, the initial distribution is assumed to be independent of the Poisson processes. We define a continuous-time stochastic evolution $\{ \eta_t \}_{t \in \R_{\ge 0}}$, $\eta_t \in \mathfrak C$, by applying $W_{(z,z+1)}$ at each time $t$ that is an event of the Poisson process $\mathcal P (z)$. More explicitly, particle at $z$ exchanges its position with particle at site $z+1$ if its color has a lower value. It is readily shown via standard techniques that under our assumptions such a random process is well-defined, see~\cite{Har78,Har72,Hol70,Lig72}.

We will need two versions of the continuous-time multi-species TASEP. Both of them will start with the initial configuration
$\id$. The first version $\{\hat \eta_{t;f} (z) \}$ is obtained from the previous construction by adding a constraint: we suppress all jumps which involve positions greather or equal to $f(t)$. In other words, this is a multi-species TASEP with a wall which moves according to a speed profile $f(t)$. The second version $\{\eta_{t;f,T} (z) \}$ is defined only for times $t \in [0,T]$, and in the same way, but the wall moves now according to the speed profile $f(T-t)$. Note that for the first process the wall moves to the right, and for the second process it moves to the left due to our assumption that $f$ is monotonously non-decreasing.

\begin{prop}\label{prop:inverTime}
For any $T \in \R_{\ge 0}$ we have
\begin{equation}
\{ \mathrm{inv} \left( \hat \eta_{T;f} \right) (z) \}_{z \in \mathbb{Z}} \stackrel{d}{=} \{\eta_{T;f,T} (z) \}_{z \in \mathbb{Z}}.
\end{equation}
\end{prop}
\begin{proof}
Directly follows from Proposition~\ref{prop:sym-basic}. Note that the inversion of time plays a crucial role.
\end{proof}

Now we are ready to prove Proposition~\ref{PropStartingIdentity}.

\begin{proof}[Proof of Proposition~\ref{PropStartingIdentity}]
Note that for any $n \geq 1$ we have
\begin{equation}
\Pb \left( x_n^f (T) > s \right) = \Pb \left( \exists \mbox{ at least $n$ numbers $i \le 0 : \mathrm{inv} \left( \hat \eta_{T;f} \right) (i) > s$} \right).
\end{equation}
Indeed, note that the process $\hat \eta_{T;f} (z)$ can be coupled with a single species TASEP in the following way: we consider colors $\le 0$ as particles, and colors $> 0$ as holes. Both sides of the equality above represent the same event under such a coupling.

Next, Proposition~\ref{prop:inverTime} implies that
\begin{multline*}
\Pb \left(\exists \mbox{ at least $n$ numbers $i \le 0 : \mathrm{inv} \left( \hat \eta_{T;f} \right) (i) > s$} \right)
\\
= \Pb \left( \exists\mbox{ at least $n$ numbers $i \le 0 : \eta_{T;f,T} (i) > s$} \right)
\end{multline*}
In words, the right-hand side contains the probability of the event that at the end of the process $\eta_{T;f,T}$ there are at least $n$ particles of color $>s$ inside the interval $(-\infty,0]$. This event can also be observed through the coupling with a single-species TASEP: we need to view all colors $> s$ as holes, and all colors less or equal than $s$ as particles. Let us apply the standard particle-hole duality for this single-species process, and change the coordinates $z \mapsto s-z$. We obtain the single-species TASEP (notation for its particle positions: $y_1 (t) > y_2 (t) > \dots$) which starts from the step initial condition $y_k(0) = -k+1$, for all $k \ge 1$, and the updates in this process are allowed only strictly \textit{to the right} of the wall, which moves according to profile $s - f(T-t)$. Thus, we get
\begin{equation}
\Pb \left(\exists \mbox{ at least $n$ numbers $i \le 0 : \eta_{T,f,T} (i) > s$} \right) = \Pb \left( y_n (T) > s \right).
\end{equation}
It remains to show that
\begin{equation}
\Pb \left( y_n (T) > s \right) =  \Pb \left( x_n (t) > s - f \left( T-t \right)\, \forall t \in [0,T] \right).
\end{equation}
In order to prove this, let us couple the processes associated with particles $\{ x_k (t) \}$ and $\{ y_k (t) \}$ via the basic coupling in the interval $(s-f(T-t), +\infty)$ for any $t \in [0,T]$. All $\{ y_k \}$ particles which happen to be to the left of $s-f(T-t)$ at some time $t$ stop forever, since $s-f(T-t)$ is a monotonously non-decreasing function. All $\{ x_k \}$ particles continue to jump at all positions, however, the evolution of particles with larger numbers does not affect the evolution of particles with smaller numbers. Thus, under such a coupling, we have
\begin{itemize}
\item[(a)] if $x_n (t) > s - f \left( T-t \right) \forall t \in [0,T]$, then
\begin{equation}
y_n (t) = x_n (t), \forall t \in [0,T],
\end{equation}
\item[(b)] if for some $t_0\in [0,T]$, $x_n (t_0) \leq s - f \left( T-t_0 \right)$, then (since $f(0)= 0$)
\begin{equation}
y_n(T) = y_n (t_0) \leq s - f(T-t_0) \le s.
\end{equation}
\end{itemize}
Therefore, the event $\{ y_n (T) > s \}$ happens if and only if the event $\{ x_n (t) > s - f \left( T-t \right) \mbox{ for all }t \in [0,T]\}$ happens.
This concludes the proof.
\end{proof}

\section{Scaling limit}\label{SectScalingLimits}

\subsection{Single region away from the boundary time}\label{SectBulkCase}
Consider the following setting. Let us fix an $\alpha\in (0,1)$ and consider, in the TASEP with step initial condition, the position of particle number $n=\alpha T$ at time $T$. Without a wall in front of the first particle, we would have $x_{\alpha T}(T)=(1-2\sqrt{\alpha})T+\Or(T^{1/3})$. If we have an influence of the wall around the time $\alpha_0 T$ for some $\alpha_0\in(\alpha,1)$ so that the effect is visible on a macroscopic scale, then we have to look at a scaling $x^f(T)\simeq \xi T$ for some $\xi<1-2\sqrt{\alpha}$. Of course, since we want to still have some fluctuations, particle number $\alpha T$ should have already moved at time $T$, i.e., we also have $\xi>-\alpha$.

Let us fix some $\alpha_0\in (\alpha,1)$. The case $\alpha_0=1$ is discussed in Section~\ref{SectBoundaryCase}. From the law of large numbers, we have (see Lemma~\ref{lemLocalStep})
\begin{equation}\label{eqLLN}
x_{\alpha T}(\alpha_0 T)\simeq \left\{\begin{array}{ll} \sqrt{\alpha_0} (\sqrt{\alpha_0}-2\sqrt{\alpha}) T, & \alpha_0\in [\alpha, 1],\\ -\alpha T,& \alpha_0\in [0,\alpha),\end{array}\right.
\end{equation}
where the second case is simply because particle number $\alpha T$ start moving only around time $\alpha T$.
The rescaled process around time $t=\alpha_0 T$ is given by
\begin{equation}\label{eqRescStep}
\widetilde X_T(\tau):=\frac{x_{\alpha T}(\alpha_0 T- \tilde c_2 \tau T^{2/3})-\tilde \mu(\tau,T)}{-\tilde c_1 T^{1/3}},
\end{equation}
where the constants $\tilde c_1,\tilde c_2$ are given by
\begin{equation}\label{eqConstantsTilde}
\tilde c_1=\frac{(\sqrt{\alpha_0}-\sqrt{\alpha})^{2/3}\alpha_0^{1/6}}{\alpha^{1/6}},\quad \tilde c_2=\frac{2(\sqrt\alpha_0-\sqrt\alpha)^{1/3}\alpha_0^{5/6}}{\alpha^{1/3}}
\end{equation}
and the function $\tilde \mu(\tau,T)$ by
\begin{equation}
\tilde\mu(\tau,T)=\sqrt{\alpha_0}(\sqrt{\alpha_0}-2\sqrt{\alpha}) T-2\tau \frac{(\sqrt\alpha_0-\sqrt\alpha)^{4/3}\alpha_0^{1/3}}{\alpha^{1/3}} T^{2/3}.
\end{equation}
Note the relation $\tilde c_2=2\tilde c_1^2 \sqrt{\alpha_0}/(\sqrt{\alpha_0}-\sqrt{\alpha})$.

For any fixed $\alpha_0\in(\alpha,1]$, a simple rescaling of the result of Proposition~\ref{propTightness} implies the following weak convergence result.
\begin{cor}\label{CorWeakConvergence}Let $\alpha_0\in (\alpha,1]$ be fixed. Then, we have
\begin{equation}\label{eq1.62}
\lim_{T\to\infty} \widetilde X_T(\tau)={\cal A}_2(\tau)-\tau^2.
\end{equation}
The convergence is the weak convergence in the space of continuous functions on compact intervals.
\end{cor}

We study a case where the wall has a non-trivial influence. More precisely, we will consider the situation where in the expression of the r.h.s.~of \eqref{eqStartingPoint} the influence is restricted to a $T^{2/3}$ neighborhood of $t=\alpha_0 T$. As we will explain below, this is the case when the function $f$ describing the evolution of the wall satisfies the assumption below.

For any given $\xi\in [-\alpha,1-2\sqrt{\alpha}]$, consider the threshold function
\begin{equation}\label{eqF0}
f_0(\beta)=\left\{
\begin{array}{ll}
\xi-\sqrt{1-\beta}(\sqrt{1-\beta}-2\sqrt{\alpha}),&\beta\in [0,1-\alpha),\\
\xi+\alpha,&\beta\in[1-\alpha,1].
\end{array}
\right.
\end{equation}
see Figure~\ref{FigThresholdFunction}. The parameter $\xi$ will give us the macroscopic position of $x^f_{\alpha T}(T)$.
\begin{figure}[t!]
\begin{center}
\psfrag{t/T}[cb]{$t/T$}
\psfrag{f0}[lc]{$f_0(t/T)$}
\psfrag{a}[cc]{$1-\alpha$}
\psfrag{a0}[cc]{$1-\alpha_0$}
\psfrag{1}[cc]{$1$}
\includegraphics[height=6cm]{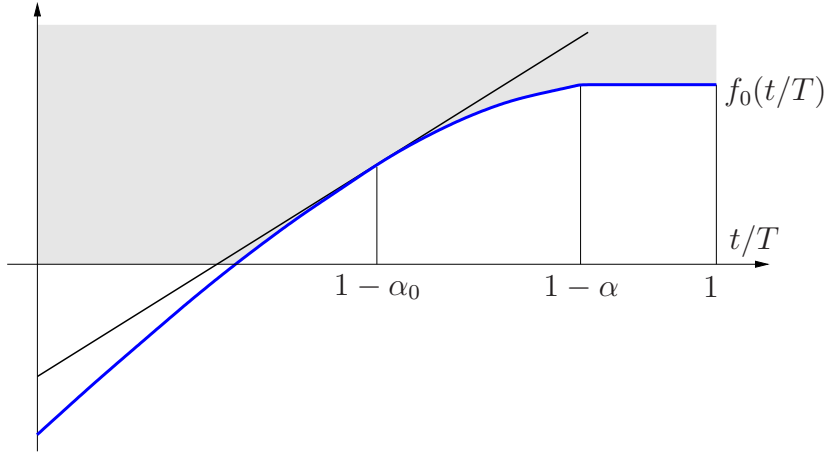}
\caption{The lower bound function $f_0$ (here for $\alpha=0.4$ und $\xi=-0.05$). In Assumption~\ref{AssumptionOnePoint} the function $f(t)/T$ has to be in the gray shaded region since $f(t)\geq 0$ as well. Depicted in black is the tangent line at $1-\alpha_0$ (here $\alpha_0=0.5$), namely $\xi-\sqrt{\alpha_0}(\sqrt{\alpha_0}-2\sqrt{\alpha})+(1-\sqrt{\alpha/\alpha_0})(t/T-1+\alpha_0)$.}
\label{FigThresholdFunction}
\end{center}
\end{figure}

\begin{assumption}\label{AssumptionOnePoint}
Assume that a non-decreasing function $f$ on $\R_{\geq 0}$ with $f(0)=0$ satisfies:
\begin{itemize}
\item[(a)] For some fixed $\e>0$, for all $|t-\alpha_0 T|> \e T$,
\begin{equation}\label{eq1.70}
f(t)\geq T f_0(t/T)+K(\e)T
\end{equation}
for some positive constant $K(\e)>0$.
This means that macroscopically away from $\alpha_0 T$, the function has to be macroscopically larger than $T f_0$.
\item[(b)] For $|t-\alpha_0 T|\leq \e T$, let us parameterize $T-t$ as $T-t=(1-\alpha_0)T+\tilde c_2\tau T^{2/3}$. Then write
\begin{equation}\label{eq1.69}
f(T-t) = \xi T - \tilde\mu(\tau,T)-\tilde c_1(\tau^2-g_T(\tau)) T^{1/3},\quad \tau\in\R.
\end{equation}
Assume that $\tau\mapsto g_T(\tau)$ converges uniformly on compact sets to a piecewise continuous\footnote{With piecewise continuous function we mean that for any bounded interval $I$ can be decomposed into finitely many subsets $I=[a_1,a_2)\cup [a_2,a_3)\cup \ldots \cup [a_{n-1},a_n]$ on which $g$ is a continuous function having one-sided limits at the end-points of the intervals of the decomposition of $I$.} function
$g$. We also assume that there exists a constant $M$ such that, for all $T$ large enough,
\begin{equation}\label{eq1.72}
g_T(\tau)\geq -M+\tau^2/2
\end{equation}
for $|\tau|\leq \e \tilde c_2^{-1} T^{1/3}$.
\end{itemize}
\end{assumption}
\begin{rem}
The condition $f(0)=0$ is easily satisfied in (a) at $t=0$ by choosing $K(\e)\leq -f_0(0)$ as $f_0(0)<0$.
\end{rem}

\begin{thm}\label{ThmOnePointSituation}
Let $f$ satisfy Assumption~\ref{AssumptionOnePoint} and recall that $x^f(t)$ is the TASEP with step initial condition and with a blocking wall moving according to $t\mapsto f(t)$. Then
\begin{equation}\label{eqStartingPointB}
\lim_{T\to\infty}\Pb(x^f_{\alpha T}(T)\geq \xi T-S \tilde c_1 T^{1/3})=\Pb\Big(\sup_{\tau\in \R}\{{\cal A}_2(\tau) -g(\tau)\}\leq S\Big),
\end{equation}
with $\tilde c_1=\frac{(\sqrt{\alpha_0}-\sqrt{\alpha})^{2/3}\alpha_0^{1/6}}{\alpha^{1/6}}$.
\end{thm}

\begin{remark}
We assumed that $g_T$ converges uniformly on compact sets to a piecewise continuous function $g$. It is not hard to allow the limiting function $g$ to be equal to infinity (the wall is removed) in some intervals since the process $\widetilde X_T$ is tight. In this case, the supremum in \eqref{eqStartingPointB} will be restricted to a subset of $\R$.
\end{remark}

Before proving Theorem~\ref{ThmOnePointSituation}, let us comment on Assumption~\ref{AssumptionOnePoint}(a). It is based on Proposition~\ref{PropStartingIdentity} and the law of large numbers \eqref{eqLLN}.  In order to have a non-trivial interaction of particle $\alpha T$ with the wall at time $\alpha_0 T$ in the r.h.s.~of \eqref{eqStartingPoint}, $f$ macroscopically needs to satisfy
\begin{equation}\label{eq3.11}
\sqrt{\alpha_0}(\sqrt\alpha_0-2\sqrt\alpha) T \simeq \xi T-f((1-\alpha_0)T).
\end{equation}
On the other hand, the influence of the wall in the r.h.s.~of \eqref{eqStartingPoint} at time $\beta T$ will be asymptotically vanishing if, for $\beta\leq\alpha$,
\begin{equation}\label{eq3.2}
-\alpha T> \xi T-f((1- \beta)T),
\end{equation}
and for $\beta\in (\alpha,1]\setminus\{\alpha_0\}$,
\begin{equation}\label{eq3.3}
\sqrt{\beta}(\sqrt{\beta}-2\sqrt\alpha) T > \xi T-f((1-\beta)T).
\end{equation}
\eqref{eq3.2} and \eqref{eq3.3} give the macroscopic lower threshold $f_0$ for $f$.

In the proof of Theorem~\ref{ThmOnePointSituation} we will have to control the increments of the rescaled process $\widetilde X_T(\tau)$. This will be done by the comparison result (Theorem~\ref{thmComparison}). Thus, as in \eqref{eqRescStep}, we will define the rescaled process for the stationary TASEP with density $\rho$ by
\begin{equation}\label{eqRescStat}
\widetilde B^{\rho}_T(\tau):=\frac{x^\rho_{N}(\alpha_0 T- \tilde c_2 \tau T^{2/3})-\tilde \mu(\tau,T)}{-\tilde c_1 T^{1/3}}.
\end{equation}
Note that since we are going to use \eqref{eqRescStat} only to bound increments  of the form $\widetilde B^\rho_T(\tau_2)-\widetilde B^\rho_T(\tau_1)$, these do not depends on the index $N$.

\subsection{Single region at the boundary time}\label{SectBoundaryCase}
This time consider $\alpha_0=1$ and $\xi=1-2\sqrt{\alpha}$. Consider the threshold function
\begin{equation}\label{eqF0bis}
f_0(\beta)=\left\{
\begin{array}{ll}
1-2\sqrt{\alpha}-\sqrt{1-\beta}(\sqrt{1-\beta}-2\sqrt{\alpha}),&\beta\in [0,1-\alpha),\\
(1-\sqrt{\alpha})^2,&\beta\in[1-\alpha,1].
\end{array}
\right.
\end{equation}

\begin{assumption}\label{AssumptionOnePointBis}
Assume that a non-decreasing function $f$ on $\R_{\geq 0}$ with $f(0)=0$ satisfies:
\begin{itemize}
\item[(a)] For some fixed $\e>0$, for all $t>\e T$,
\begin{equation}\label{eq1.70Bis}
f(t)\geq T f_0(t/T)+K(\e)T
\end{equation}
for some positive constant $K(\e)>0$.
\item[(b)] For $t\leq \e T$, let us parameterize $T-t$ as $T-t=\tilde c_2\tau T^{2/3}$. Then write
\begin{equation}\label{eq1.69Bis}
f(T-t) = (1-2\sqrt{\alpha})T - \tilde\mu(\tau,T)-\tilde c_1(\tau^2-g_T(\tau)) T^{1/3},\quad \tau\in\R_+.
\end{equation}
Assume that $\tau\mapsto g_T(\tau)$ converges uniformly on compact sets to a piecewise continuous function $g$. We also assume that there exists a constant $M$ such that, for all $T$ large enough,
\begin{equation}\label{eq1.72Bis}
g_T(\tau)\geq -M+\tau^2/2
\end{equation}
for $|\tau|\leq \e \tilde c_2^{-1} T^{1/3}$.
\end{itemize}
\end{assumption}

\begin{thm}\label{ThmOnePointSituationBis}
Let $f$ satisfy Assumption~\ref{AssumptionOnePointBis} and recall that $x^f(t)$
is the TASEP with the step initial condition and with blocking wall moving
according to $t\mapsto f(t)$. Then
\begin{equation}\label{eqStartingPointBBis}
\lim_{T\to\infty}\Pb(x^f_{\alpha T}(T)\geq \xi T-S \tilde c_1
T^{1/3})=\Pb\Big(\sup_{\tau\in \R_+}\{{\cal A}_2(\tau) -g(\tau)\}\leq S\Big),
\end{equation}
with $\tilde c_1=\frac{(\sqrt{\alpha_0}-\sqrt{\alpha})^{2/3}\alpha_0^{1/6}}{\alpha^{1/6}}$.
\end{thm}
The proof of Theorem~\ref{ThmOnePointSituationBis} is a slight simplification of the one of Theorem~\ref{ThmOnePointSituation}. The simplification is that, for the in the r.h.s.~of \eqref{eqStartingPoint}, instead of controlling the bound away from the macroscopic time $\alpha_0 T$ on both sides, here we have only one side.

\subsection{Special case: constant speed}\label{secConstantSpeed}
In this section we deal with a special case of the wall moving with a constant speed $v$. Let $\mathrm{F}_2$ and $\mathrm{F}_1$ be the GUE- and GOE-Tracy-Widom distributions, and let $\mathrm{F}_{2 \to 1;0}$ be the crossover distribution given by the time distribution of the $\mathcal{A}_{2 \to 1}$ at time $0$ introduced in~\cite{BFS07}.
For such a choice of the behaviour of the wall, the limiting distributions of particles are the following.
\begin{prop}\label{PropLinearCase}
Let $f(0)=0$ and $f(t)=cT + v t$ for  $t>0$, with $c \ge 0$, $v\in (0,1)$. \\
(a) For $c<1-v$ and $\alpha<(1-v-\sqrt{c(1-v)})^2$, we have
\begin{equation}
\lim_{T \to \infty} \Pb\left(x_{\alpha T}^f (T)\geq
\Big(v+c-\frac{\alpha}{1-v}\Big) T-S \tilde c_1 T^{1/3} \right) =
\mathrm{F}_1(2^{2/3} S),
\end{equation}
with $\tilde c_1=\alpha^{1/3}v^{2/3}/(1-v)$.\\
(b) For $c=0$ and $\alpha=(1-v)^2$ we have
\begin{equation}
\lim_{T \to \infty}  \Pb\left(x_{\alpha T}^f (T) \geq (1-2 \sqrt{\alpha}) T -S \tilde c_1 T^{1/3}\right)= \mathrm{F}_{2 \to 1;0}(S),
\end{equation}
with $\tilde c_1=v^{2/3}/(1-v)^{1/3}$.\\
(c) For  $c\geq 0$ and $\alpha>(1-v-\sqrt{c(1-v)})^2$ we have
\begin{equation}
\lim_{T \to \infty} \Pb\left(x_{\alpha T}^f (T) \geq (1-2 \sqrt{\alpha}) T -S \tilde c_1 T^{1/3}\right) = \mathrm{F}_2(S),
\end{equation}
with $\tilde c_1=v^{2/3}/(1-v)^{1/3}$.
\end{prop}

\begin{remark}
The case $\alpha=(1-v-\sqrt{c(1-v)})^2$ for $c>0$ corresponds to a shock-type situation, which we do not address in this paper. The appearance of the $\mathrm{F}_{2 \to 1;0}$ distribution in case (b) of Proposition~\ref{PropLinearCase} can be understood by comparing the characteristic lines, see Figure~\ref{FigCharacteristics}.
\end{remark}

\begin{figure}[t!]
\begin{center}
\psfrag{x}[lc]{$x$}
\psfrag{t}[lb]{$t$}
\psfrag{a}[c]{(a)}
\psfrag{b}[c]{(b)}
\includegraphics[height=4cm]{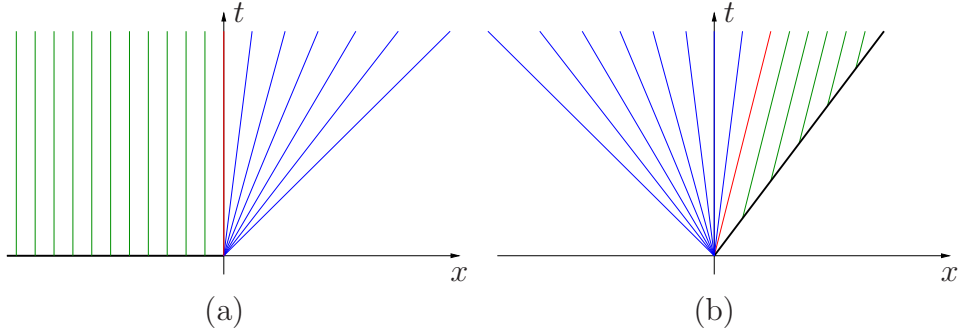}
\caption{Characteristic lines for: (a) half-flat initial condition: $x_k(0)=-2k$, $k\geq 0$, and (b) step initial condition with wall starting at the origin with constant speed smaller than $1$, i.e., case (b) of Proposition~\ref{PropLinearCase}. One sees that geometrically are very similar and indeed the fluctuations around the red line are both given in terms of the Airy$_{2\to 1}$ process.}
\label{FigCharacteristics}
\end{center}
\end{figure}

\begin{proof}[Proof of Proposition~\ref{PropLinearCase}]
Linearly expanding $x_{\alpha T}(t)$ and $\xi T-f(T-t)$ around \mbox{$t=\alpha_0 T$}L we get that, c.f.\ \eqref{eq3.11},
\begin{equation}
v=1-\sqrt{\alpha/\alpha_0},\quad \xi=c+v-\alpha/(1-v).
\end{equation}
Since the function $f$ is increasing and $f(t)\geq 0$ for all time, we necessarily have to take $v\geq 0$ and $c\geq 0$. Furthermore, $v>0$ comes from the requirement that $\alpha_0>\alpha$.

Case (a), $\alpha<(1-v-\sqrt{c(1-v)})^2$, corresponds to the case $\alpha_0<1$ together with the requirement that $\xi< 1-2\sqrt{\alpha}$. In this case Assumption~\ref{AssumptionOnePoint} is satisfied with $\xi = v+c-\frac{\alpha}{1-v}$, $\alpha_0=\alpha/(1-v)^2$, and $g_T(\tau) = \tau^2$. Therefore, applying Theorem~\ref{ThmOnePointSituation} we obtain
\begin{equation}
\lim_{T\to\infty}\Pb \left(x^f_{\alpha T}(T)\geq \Big(v+c-\frac{\alpha}{1-v}\Big) T-S \tilde c_1 T^{1/3} \right)=\Pb\Big(\sup_{\tau\in \R}\{{\cal A}_2(\tau) -\tau^2 \}\leq S\Big)
\end{equation}
It is known that the distribution given by the right-hand side is the Tracy-Widom $\mathrm{F}_1$ distribution rescaled by $2^{1/3}$. This result was first obtained in~\cite[Corollary~1.3]{Jo03b}.

Case (b), $\alpha=(1-v)^2$ and $c=0$ corresponds to $\alpha_0=1$ and thus $\xi=1-2\sqrt{\alpha}$. Assumption~\ref{AssumptionOnePointBis} is satisfied with $g_T(\tau) = \tau^2$. Therefore, we can apply Theorem~\ref{ThmOnePointSituationBis} and obtain
\begin{equation}
\lim_{T\to\infty}\Pb\left(x^f_{\alpha T}(T)\geq (1-2 \sqrt{\alpha}) T-S \tilde c_1 T^{1/3}\right)=\Pb\Big(\sup_{\tau\in \R_+}\{{\cal A}_2(\tau) -\tau^2 \}\leq S\Big),
\end{equation}
The limiting distribution in the right-hand side was shown to be equal to $\mathrm{F}_{2 \to 1;0}$ in~\cite{QR13b}.

Case (c), $\alpha>(1-v-\sqrt{c(1-v)})^2$ is not covered by the above theorems. However, this case is much easier than the previous two since it corresponds to the region whose limit behavior is not affected by the moving wall. We have not included a general theorem for this situation, so let us present a brief argument in this special case. Applying \eqref{eqStartingPoint} and the law of large numbers \eqref{eqLLN} we obtain that in this case the critical constraint in estimating
\begin{equation}
\label{eqFirstRepeat}
\Pb(x_{\alpha T} (t)> s-f(T-t)\textrm{ for all }t\in [0,T])
\end{equation}
comes from the case $t=T$ (recall that $f(0)=0$ and the function $f$ is not continuous at $0$ if $c >0$), while for other $t \in [0,T)$ and $s= (1-2\sqrt{\alpha})T + S \tilde c_1 T^{1/3}$ the condition $x_n(t)> s-f(T-t)$ is satisfied for every single moment of time with a macroscopic difference between the two sides with probability exponentially close to 1, and this can be extended to an estimate which is uniform in $t$ in the same way as in Lemma~\ref{lemmaMacroPoints} below. Thus, the limit of probability from \eqref{eqFirstRepeat} is equal to the limit of
\begin{equation}
\Pb(x_{\alpha T} (T)> s),
\end{equation}
which is given by the Tracy-Widom $\mathrm{F}_2$ distribution via a classical result of Johansson~\cite{Jo00b} (which is a single-point version of \eqref{eq1.62} above).
\end{proof}

\begin{remark}
Modifying (b) slightly we can get the distribution of the Airy$_{2\to 1}$ at a different time than $0$. For instance, take $c=0$, $\xi=1-2\sqrt{\alpha}$, $\alpha_0=1$ but now $v=1-\sqrt{\alpha}+\delta (1-\sqrt{\alpha})^{1/3}\alpha^{1/6} T^{-1/3}$ and $g_T(\tau)=\tau^2-2\delta \tau$ for some $\delta\geq 0$. Then we get from Theorem~\ref{ThmOnePointSituationBis}
\begin{equation}
\begin{aligned}
\lim_{T\to\infty}\Pb\left(x^f_{\alpha T}(T)\geq (1-2 \sqrt{\alpha}) T-S \tilde c_1 T^{1/3}\right)&=\Pb\Big(\sup_{\tau\in \R_+}\{{\cal A}_2(\tau) -\tau^2+2\delta\tau \}\leq S\Big)\\
&=\Pb\Big(\sup_{u\leq \delta}\{{\cal A}_2(u) -u^2 \}\leq S-\delta^2\Big),
\end{aligned}
\end{equation}
where we made the change of variable $\tau = \delta-u$ and used the stationarity of the Airy$_2$ process, namely ${\cal A}_2(\delta-u)\stackrel{(d)}{=}{\cal A}_2(u)$.
By Theorem~1 of~\cite{QR13b} we have
\begin{equation}
\begin{aligned}
\Pb\Big(\sup_{u\leq \delta}\{{\cal A}_2(u) -u^2 \}\leq S-\delta^2\Big) &= \Pb({\cal A}_{2\to 1}(\delta)\leq S-\max\{0,\delta\}^2)\\
&=\mathrm{F}_{2 \to 1;\delta}(S-\max\{0,\delta\}^2).
\end{aligned}
\end{equation}
\end{remark}

\subsection{Asymptotics}
Now we are going to prove Theorem~\ref{ThmOnePointSituation}.
\begin{proof}[Proof of Theorem~\ref{ThmOnePointSituation}]
Let $n=\alpha T$ and consider $s=\xi T-S \tilde c_1 T^{1/3}$ and $t=\alpha_0 T-\tilde c_2\tau T^{2/3}$ with constants from \eqref{eqConstantsTilde}. Then by Proposition~\ref{PropStartingIdentity} we have
\begin{equation}
\begin{aligned}
&\Pb(x^f_{\alpha T}(T)\geq \xi T-S \tilde c_1 T^{1/3})\\
&=\Pb\big(x_{\alpha T}(t)\geq \xi T-S \tilde c_1 T^{1/3}-f(T-t),\,\forall\, t\in[0,T]\big)\\
&=\Pb\big(x_{\alpha T}(\alpha_0 T-\tilde c_2\tau T^{2/3})\geq \xi T-S\tilde c_1 T^{1/3}-f((1-\alpha_0)T+\tilde c_2\tau T^{2/3}),\, \forall \,\tau\in I_{T}\big)\\
&=\Pb\Big(\widetilde X_T(\tau)\leq S+\frac{\xi T-f((1-\alpha_0)T+\tilde c_2\tau T^{2/3})-\tilde \mu(\tau,T)}{-\tilde c_1 T^{1/3}},\,\forall\, \tau\in I_T\Big),
\end{aligned}
\end{equation}
where $I_T=[-(1-\alpha_0) \tilde c_2^{-1}T^{1/3},\alpha_0 \tilde c_2^{-1}T^{1/3}]$.

Under \eqref{eq1.70} we will have that for $|\tau|\gg 1$ the condition is satisfied with high probability. Under \eqref{eq1.69} we then get that the contribution for $\tau$ of order $1$ will give
\begin{equation}
\Pb(\widetilde X_T(\tau)\leq S-\tau^2+g_T(\tau),\,\forall\,\tau=\Or(1)),
\end{equation}
which due to \eqref{eq1.62} will lead, in the $T\to\infty$ limit, to
\begin{equation}
\Pb({\cal A}_2(\tau)\leq S+g(\tau),\,\forall\, \tau\in\R).
\end{equation}

Let $\ell_0\in \N$ be a number to be chosen later (independently of $T$). Since we assumed that $g_T$ converges to $g$ uniformly on compact sets, for any $\e>0$, $\sup_{\tau\in [-\ell_0 \varkappa,\ell_0 \varkappa]} |g_T(\tau)-g(\tau)|\leq \e$ for all $T$ large enough. Therefore, for any interval $[a,b]$ where $g$ is continuous, as $\widetilde X_T(\tau)$ is tight in the space of continuous functions on it, then also $\widetilde X_T(\tau)+\tau^2-g_T(\tau)$ is tight in the same space as well and, by Corollary~\ref{CorWeakConvergence}, we have
\begin{equation}\label{eq3.28}
\sup_{\tau\in [a,b]}\{\widetilde X_T(\tau)+\tau^2-g_T(\tau)\}\Rightarrow \sup_{\tau\in [a,b]}\{{\cal A}_2(\tau) -g(\tau)\}.
\end{equation}
In our case, for any $\ell_0>0$, the interval $[-\ell_0\varkappa,\ell_0\varkappa]$ can be decomposed into the union of intervals $[a_1,a_2]$, $(a_2,a_3]$, $\ldots$, $(a_{n-1},a_n]$ such that \eqref{eq3.28} holds jointly for all these intervals, since the supremum is a continuous function in its (finitely many) arguments. From this we get
\begin{equation}
\lim_{T\to\infty} \Pb\Big(\sup_{\tau\in [-\ell_0 \varkappa,\ell_0 \varkappa]}\{\widetilde X_T(\tau)+\tau^2-g_T(\tau)\}\leq S\Big) = \Pb\Big(\sup_{\tau\in [-\ell_0 \varkappa,\ell_0 \varkappa]}\{{\cal A}_2(\tau) -g(\tau)\}\leq S\Big).
\end{equation}
The reason to consider a region of size $\ell_0 \varkappa$ and not simply $\varkappa$ is that later we will control the increments in pieces of size $\varkappa$ and apply a rescaled version of Theorem~\ref{thmComparison}. Notation-wise we found it simpler.

Furthermore, if $g(\tau)\geq c \tau^2-M_1$ for some constants $M_1$ and $c\in (1/4,1]$ then by Theorem~3.2 (see also Theorem 1.19) of~\cite{QR16}
\begin{equation}
\lim_{\varkappa\to\infty} \Pb\Big(\sup_{\tau\in [-\ell_0 \kappa,\ell_0 \varkappa]}\{{\cal A}_2(\tau) -g(\tau)\}\leq S\Big) = \Pb\Big(\sup_{\tau\in\R}\{{\cal A}_2(\tau) -g(\tau)\}\leq S\Big).
\end{equation}
Thus, to prove Theorem~\ref{ThmOnePointSituation} we need to show that
\begin{equation}\label{eq1.60}
\lim_{\varkappa\to\infty} \lim_{T\to\infty} \Pb\Big(\sup_{|\tau|\geq \ell_0 \varkappa}\{\widetilde X_T(\tau)+\tau^2-g_T(\tau)\}\leq S\Big)=1,
\end{equation}
which in the unscaled version reads
\begin{equation}
\lim_{\varkappa\to\infty} \lim_{T\to\infty} \Pb\big(x_{\alpha T}(t)\geq \xi T-S \tilde c_1 T^{1/3}-f(T-t),\,\forall\, t\in J_T\big)=1
\end{equation}
where $J_T=[0,T]\setminus[\alpha_0T-\varkappa\ell_0\tilde c_2 T^{2/3},\alpha_0T+\varkappa\ell_0\tilde c_2 T^{2/3}]$.

Combining Lemmas~\ref{lemmaVerySmallTimes},~\ref{lemmaSmallTimes},~\ref{lemmaMacroPoints},~\ref{lemModerate}, to be proven below, we get the following result:
\begin{equation}
\lim_{T\to\infty} \Pb\Big(\sup_{|\tau|\geq \ell_0 \varkappa}\{\widetilde X_T(\tau)+\tau^2-g_T(\tau)\}\leq S\Big)\geq 1-C e^{-c\varkappa}
\end{equation}
for some constants $C,c$. Taking $\varkappa\to\infty$ we reach \eqref{eq1.60}.
\end{proof}

\subsubsection{Bounds for regions macroscopically away from $t=\alpha_0 T$.}

\textbf{Small times easy bound.} Consider first the simple case $t\in [0,\beta_0 T]$ with \mbox{$\beta_0=\min\{\alpha,\alpha_0-\e\}$}. This corresponds to the case where particle $x_{\alpha T}(t)$ did not really move from its initial position $-\alpha T$.
\begin{lem}\label{lemmaVerySmallTimes}
Let $t\in [0,\min\{\alpha,\alpha_0-\e\} T]$. Then, for all $T$ large enough,
\begin{equation}
\Pb(x_{\alpha T}(t)\geq \xi T-S \tilde c_1 T^{1/3}-f(T-t),\,\forall\, 0\leq t\leq \beta_0 T) = 1.
\end{equation}
\end{lem}
\begin{proof}
We need to bound from below the probability
\begin{equation}
\begin{aligned}
&\Pb(x_{\alpha T}(t)\geq \xi T-S \tilde c_1 T^{1/3}-f(T-t),\,\forall\, 0\leq t\leq \beta_0 T)\\
&=1-\Pb(x_{\alpha T}(t)< \xi T-S \tilde c_1 T^{1/3}-f(T-t)\,\textrm{for some}\, 0\leq t\leq \beta_0 T).
\end{aligned}
\end{equation}
Let us denote $t=\beta T$. Then, \eqref{eq1.70} and \eqref{eqF0} give
\begin{equation}
\xi T-S \tilde c_1 T^{1/3}-f(T-t)\leq -\alpha T-K(\e) T-S \tilde c_1 T^{1/3},
\end{equation}
which for all fixed $\e>0$ (and thus fixed $K(\e)$) and $S$, is strictly smaller than $-\alpha T$ for all $T$ large enough. But $x_{\alpha T}(t)\geq -\alpha T$ for all times. This implies our claim.
\end{proof}

\medskip
\noindent \textbf{Small time interval when particle number $\alpha T$ could move.} Consider now slightly larger times, for which we can not yet use the asymptotic bounds from Lemma~\ref{LemOnePointBounds}.
\begin{lem}\label{lemmaSmallTimes}
Let us consider $\delta>0$ defined by $f_0(1-\alpha)-f_0(1-\alpha-\delta)=K(\e)/2$. Then for all $T$ large enough,
\begin{equation}
\Pb(x_{\alpha T}(t)\geq \xi T-S \tilde c_1 T^{1/3}-f(T-t),\,\forall\, \alpha T\leq t\leq (\alpha+\delta) T) = 1.
\end{equation}
\end{lem}
\begin{proof}
Since $f_0$ is monotone, the assumption implies that for all $y\in [0,\delta]$, \mbox{$f_0(1-\alpha)-f_0(1-\alpha-y)\leq K(\e)/2$}. Then, for $t=\beta T$ with $\beta\in [\alpha,\alpha+\delta]$, we have
\begin{equation}
\xi T-S \tilde c_1 T^{1/3}-f(T-t)\leq -\alpha T-K(\e) T/2-S \tilde c_1 T^{1/3},
\end{equation}
and by the same argument as above we get for all $T$ large enough
\begin{equation}
\Pb(x_{\alpha T}(t)\geq \xi T-S \tilde c_1 T^{1/3}-f(T-t),\,\forall\, \alpha T\leq t\leq (\alpha+\delta) T) = 1.
\end{equation}
Notice that this case is needed only if $\alpha_0\geq \alpha+\delta$.
\end{proof}

\medskip
\noindent \textbf{Time interval $\beta T$ with $\beta\in[\alpha+\delta,\alpha_0-\e]\cup [\alpha_0+\e,1]$.}
Now we consider the rest of the time intervals which are macroscopically away from $\alpha_0 T$.
\begin{lem}\label{lemmaMacroPoints}
Denote $J=[(\alpha+\delta)T,(\alpha_0-\e)T]\cup[(\alpha_0+\e)T,T]$. Then
\begin{equation}\label{eq1.87}
\Pb(x_{\alpha T}(t)\geq \xi T-S \tilde c_1 T^{1/3}-f(T-t),\,\forall\, t\in J) \geq 1-C e^{-c\min\{K(\e)T^{2/3},T^{1/3}\}}
\end{equation}
for constants $C,c>0$ and all $T$ large enough.
\end{lem}
\begin{proof}
We compute the complementary probability and divide the time interval into pieces of length $1$. We need to find an upper bound for
\begin{equation}\label{eq1.88}
\Pb(\exists t\in[\beta T,\beta T+1] \,|\, x_{\alpha T}(t)< \xi T-S \tilde c_1 T^{1/3}-f(T-t)).
\end{equation}
In a time interval of order $1$ the evolution of any given particle is stochastically bounded by a Poisson process with intensity $1$, due the fact that it can not move faster than the case where the blocking interactions with other particles are removed. Therefore, the probability of making at least $k$ steps is bounded by $\Pb(\Xi_{\rm Poi}\geq k)$ where $\Xi_{\rm Poi}$ has Poisson distribution with parameter $1$. As $\Pb(\Xi_{\rm Poi}\geq k)\leq 4 e^{-k}$ for all $k\geq 0$,
\begin{equation}
\Pb(|x_{\alpha T}(\beta T+1)-x_{\alpha T}(\beta T)|\geq T^{1/3})\leq 4 e^{-T^{1/3}}
\end{equation}
Also, $|f(T-t)-f(T-\beta T)|=\Or(1)$ for $t\in[\beta T,\beta T+1]$. This means that
\begin{equation}\label{eq1.85}
\begin{aligned}
\eqref{eq1.88}&\leq 4e^{-T^{1/3}}+\Pb(x_{\alpha T}(\beta T)< \xi T-(S \tilde c_1-1) T^{1/3}-f(T-\beta T))\\
&\leq 4 e^{-T^{1/3}}+\Pb(x_{\alpha T}(\beta T)< \sqrt{\beta}(\sqrt{\beta}-2\sqrt{\alpha})T-K(\e)T-(S \tilde c_1-1) T^{1/3})\\
&\leq 4 e^{-T^{1/3}}+\Pb(x_{\alpha T}(\beta T)< \sqrt{\beta}(\sqrt{\beta}-2\sqrt{\alpha})T-K(\e)T/2)
\end{aligned}
\end{equation}
for all $T$ large enough, where in the first inequality we used the assumption \eqref{eq1.70}, while the second inequality holds for all $T$ large enough since $K(\e)>0$.
Applying the upper tail bound in Lemma~\ref{LemOnePointBounds} with the replacements $t\to \beta T$ and $\alpha\to \alpha/\beta$ we obtain
\begin{equation}
\eqref{eq1.85}\leq 4 e^{-T^{1/3}}+ C e^{-c K(\e) T^{2/3}}
\end{equation}
for some new constants $C,c$ (which can be chosen uniformly for $\beta\in [\alpha+\delta,1]$). Since the total number of size $1$ segments in our decomposition is $\Or(T)$, the complementary probability of \eqref{eq1.87} is bounded by $T (4 e^{-T^{1/3}}+ C e^{-c K(\e) T^{2/3}})$. Thus, for a new constant $c$ the claimed result holds true.
\end{proof}

\subsubsection{Bound in moderate deviation regimes.}
Now we turn to the delicate bounding for times in an $\e T$ neighborhood of $\alpha_0 T$. The strategy is to divide the time interval into pieces of size $\varkappa$ (in the scaled variable) and bound the process in this part by controlling the starting point as well as the fluctuations. The latter are controlled using the comparison with stationarity from Section~\ref{SectTightness}.
\begin{lem}\label{lemModerate}
There exists an $\e_0>0$ such that for all $\e\in (0,\e_0]$ and $T$ large enough,
\begin{equation}
\begin{aligned}
&\Pb(x_{\alpha T}(t)\geq \xi T-f(T-t)-S \tilde c_1 T^{1/3},\,\forall t\in[\alpha_0 T-\e T,\alpha_0 T-\tilde c_2 \ell_0 \varkappa T^{2/3}])\\
&\geq 1-C e^{-c\varkappa\ell_0^2},
\end{aligned}
\end{equation}
where $\ell_0$ is a positive number depending only on $\alpha,\alpha_0$, and $\varkappa$ is any number in $[1,\e T^{1/3}]$. The coefficients $C,c$ are uniform in $T$. Similarly,
\begin{equation}
\begin{aligned}
&\Pb(x_{\alpha T}(t)\geq \xi T-f(T-t)-S \tilde c_1 T^{1/3},\,\forall t\in[\alpha_0 T+\tilde c_2 \ell_0 \varkappa T^{2/3},\alpha_0 T+\e T])\\
&\geq 1-C e^{-c\varkappa \ell_0^2}.
\end{aligned}
\end{equation}
\end{lem}
\begin{proof}
Let us bound the first probability
\begin{equation}\label{eq1.88bis}
\Pb(x_{\alpha T}(t)\geq \xi T-f(T-t)-S \tilde c_1 T^{1/3},\,\forall t\in[\alpha_0 T-\e T,\alpha_0 T-\tilde c_2 \ell_0 \varkappa T^{2/3}]).
\end{equation}
The bound of
\begin{equation}\label{eq1.88b}
\Pb(x_{\alpha T}(t)\geq \xi T-f(T-t)-S \tilde c_1 T^{1/3},\,\forall t\in[\alpha_0 T+\tilde c_2 \ell_0 \varkappa T^{2/3},\alpha_0 T+\e T])
\end{equation}
follows along the same steps and therefore we skip the details for that case.

Using the parametrization $t=\alpha_0 T-\tilde c_2\tau T^{2/3}$, $f$ as given in \eqref{eq1.69}, and $\widetilde X_T$ as defined in \eqref{eq1.62}, we obtain
\begin{equation}
\begin{aligned}
\eqref{eq1.88bis} &= \Pb(\widetilde X_T(\tau)\leq S-\tau^2+g_T(\tau),\,\forall\tau\in[\ell_0\varkappa,\e \tilde c_2^{-1}T^{1/3}])\\
&=1-\Pb\Big(\sup_{\tau \in[\ell_0 \varkappa,\e \tilde c_2^{-1}T^{1/3}]} [\widetilde X_T(\tau)+\tau^2-g_T(\tau)]>S\Big).
\end{aligned}
\end{equation}
We decompose the time interval into pieces of width $\varkappa$ and get
\begin{equation}\label{eq1.90}
1-\eqref{eq1.88bis}\leq \sum_{\ell=\ell_0}^{\e\tilde c_2^{-1}T^{1/3}/\varkappa} \Pb\Big(\sup_{\tau\in [\varkappa\ell,\varkappa (\ell+1)]}[\widetilde X_T(\tau)+\tau^2-g_T(\tau)]>S\Big).
\end{equation}
Next we use the assumption \eqref{eq1.72} on $g_T$, namely $-g_T(\tau)\leq -\frac12\tau^2 +M$. Then we get
\begin{equation}\label{eq3.45}
\begin{aligned}
&\Pb\Big(\sup_{\tau\in [\varkappa\ell,\varkappa (\ell+1)]}[\widetilde X_T(\tau)+\tau^2-g(\tau)]>S\Big)\\
&\leq\Pb\Big(\sup_{\tau\in [\varkappa\ell,\varkappa (\ell+1)]}[\widetilde X_T(\tau)+\tfrac12\tau^2]>S-M\Big)\\
&\leq\Pb\Big(\sup_{\tau\in [\varkappa\ell,\varkappa (\ell+1)]}\widetilde X_T(\tau)> S-M-\tfrac12\varkappa^2(\ell+1)^2\Big)\\
&\leq\Pb\left(\widetilde X_T(\varkappa (\ell+1))> S-M-\tfrac34\varkappa^2(\ell+1)^2\right)\\
&\quad+\Pb\Big(\sup_{\tau\in [\varkappa\ell,\varkappa (\ell+1)]}[\widetilde X_T(\tau)-\widetilde X_T(\varkappa(\ell+1))]>\tfrac14 \varkappa^2(\ell+1)^2\Big).
\end{aligned}
\end{equation}
To estimate $\Pb(\widetilde X_T(\varkappa (\ell+1))> S-M-\tfrac34\varkappa^2(\ell+1)^2)$ we apply Lemma~\ref{LemOnePointBounds} with $t=\alpha_0 T-\tilde c_2\varkappa\ell T^{2/3}$ and $\alpha\to \tilde\alpha=\alpha T/t$. A computation gives
\begin{equation}
(1-2\sqrt{\tilde \alpha}) t = \tilde\mu(\varkappa(\ell+1),T)+\tilde c_1 \varkappa^2(\ell+1)^2 T^{1/3}(1+\Or(\varkappa\ell)).
\end{equation}
Therefore
\begin{equation}
\Pb(\widetilde X_T(\varkappa (\ell+1))> S-M-\tfrac34\varkappa^2(\ell+1)^2) = \Pb(x_{\tilde\alpha t}(t)< (1-2\sqrt{\tilde\alpha})t-\tilde s \tilde c_1 t^{1/3})
\end{equation}
with
\begin{equation}
\tilde s= S-M-\tfrac34\varkappa^2(\ell+1)^2+\varkappa^2(\ell+1)^2\left(1+\Or(\varkappa\ell t^{-1/3})\right).
\end{equation}
Since $\ell\leq  \e\tilde c_2^{-1}T^{1/3}/\varkappa$, where $\e>0$ can be chosen as small as desired (but independent of $T$), we take first $\e$ small enough so that \begin{equation}
-\tfrac34\varkappa^2(\ell+1)^2+\varkappa^2(\ell+1)^2\left(1+\Or(\varkappa\ell t^{-1/3})\right)\geq \tfrac18\varkappa^2(\ell+1)^2> \tfrac18\varkappa^2\ell^2.
\end{equation}
Then Lemma~\ref{LemOnePointBounds} gives
\begin{equation}
\Pb(\widetilde X_T(\varkappa (\ell+1))> S-M-\tfrac34\varkappa^2(\ell+1)^2) \leq C e^{-c S} e^{-c\varkappa^2\ell^2/8}.
\end{equation}

The increments over a time span $\varkappa$ of $\widetilde X_T$ (the final term of \eqref{eq3.45}) are controlled using the comparison with stationary as follows. We apply Theorem~\ref{thmComparison} with the following change of variables
\begin{equation}
\begin{aligned}
T &\to \alpha_0 T-\tilde c_2\varkappa\ell T^{2/3},\\
t=T-\varkappa T^{2/3} & \to \alpha_0 T-\tilde c_2 \varkappa (\ell+1)T^{2/3},\\
\rho_0=\sqrt{\alpha T/t}& \to \rho_0=\sqrt{\alpha T/(\alpha_0 T-\tilde c_2 \varkappa (\ell+1)T^{2/3})}.
\end{aligned}
\end{equation}
A computation gives
\begin{equation}
\rho_0=\frac{\sqrt{\alpha}}{\sqrt{\alpha_0}}+\frac12\tilde c_2 (\ell+1)\frac{\sqrt{\alpha}}{\alpha_0^{3/2}}\varkappa T^{-1/3}+\Or(\ell^2\varkappa^2 T^{-2/3}).
\end{equation}
Therefore we consider stationary TASEP with densities $\rho_\pm$ given by
\begin{equation}
\rho_\pm=\rho_0 \pm \kappa t^{-1/3} =
\frac{\sqrt{\alpha}}{\sqrt{\alpha_0}}+\Big(\frac12\tilde c_2 (\ell+1)\frac{\sqrt{\alpha}}{\alpha_0^{3/2}}\varkappa\pm \frac{\kappa}{\alpha_0^{1/3}}\Big)T^{-1/3} +\Or(\ell^2\varkappa^2,\ell \kappa)T^{-2/3}.
\end{equation}
Then, Theorem~\ref{thmComparison} tells us that for all $t_1,t_2$ such that $\alpha_0 T-\tilde c_2 \varkappa (\ell+1) T^{2/3}\leq t_1<t_2\leq \alpha_0 T-\tilde c_2 \varkappa \ell T^{2/3}$, with  probability at least $1-C e^{-c\kappa}$,
\begin{equation}
x^{\rho_+}_M(t_2)-x^{\rho_+}_M(t_1)\leq x_N(t_2)-x_N(t_1)\leq x^{\rho_-}_P(t_2)-x^{\rho_-}_P(t_1).
\end{equation}
Recall that the rescaled processes are divided by a negative number, see \eqref{eqRescStep} and \eqref{eqRescStat}. Therefore we get
\begin{multline}\label{eq1.103}
\Pb\Big(\sup_{\tau\in [\varkappa\ell,\varkappa (\ell+1)]}[\widetilde X_T(\tau)-\widetilde X_T(\varkappa(\ell+1))]>\tfrac14 \varkappa^2(\ell+1)^2\Big)
\leq C e^{-c\kappa}\\
 + \Pb\Big(\sup_{\tau\in [\varkappa\ell,\varkappa (\ell+1)]}[\widetilde B^{\rho_+}_T(\tau)-\widetilde B^{\rho_+}_T(\varkappa(\ell+1))]>\tfrac14 \varkappa^2(\ell+1)^2\Big).
\end{multline}

Let $\tilde\tau=\varkappa(\ell+1)-\tau \in [0,\varkappa]$. Then we need to bound
\begin{equation}
\sup_{\tilde \tau\in [0,\varkappa]}[\widetilde B^{\rho_+}_T(\varkappa (\ell+1)-\tilde\tau)-\widetilde B^{\rho_+}_T(\varkappa(\ell+1))].
\end{equation}
The increments of $x^{\rho_+}_M(\tau)-x^{\rho_+}_M(\varkappa(\ell+1))$ form a Poisson process with intensity $1-\rho_+$. By Lemma~\ref{LemPPtoBM} we know that $\widetilde B^{\rho_+}_T(\tau)-\widetilde B^{\rho_+}_T(\varkappa(\ell+1))$ converges to a Brownian motion with drift. A computation gives that the diffusion coefficient is $2$, while the drift is given by
\begin{equation}
\begin{aligned}
v&=\frac{\tilde\mu(\varkappa(\ell+1)-\tilde\tau,T)-\tilde\mu(\varkappa(\ell+1),T)+(1-\rho_+)\tilde c_2\tilde \tau T^{2/3}}{-\tilde c_1 T^{1/3}}\\
&=\frac{\tilde c_2}{\tilde c_1\alpha_0^{1/3}}\left(\kappa + \frac{2 (\sqrt{\alpha_0}-\sqrt{\alpha})^{2/3}\varkappa(\ell+1)}{\alpha^{1/3}}\right)+\Or(\varkappa\kappa\ell,\varkappa^2\ell^2)T^{-1/3}.
\end{aligned}
\end{equation}
Therefore we get
\begin{equation}\label{eq1.106}
\begin{aligned}
&\Pb\Big(\sup_{\tau\in [\varkappa\ell,\varkappa (\ell+1)]}[\widetilde B^{\rho_+}_T(\tau)-\widetilde B^{\rho_+}_T(\varkappa(\ell+1))]>\tfrac14 \varkappa^2(\ell+1)^2\Big)\\
&\leq \Pb\Big(\sup_{\tilde \tau\in [0,\varkappa]}[\widetilde B^{\rho_+}_T(\varkappa (\ell+1)-\tilde\tau)-\widetilde B^{\rho_+}_T(\varkappa(\ell+1))-v\tilde \tau]>\tfrac14 \varkappa^2(\ell+1)^2-v\varkappa\Big)\\
&=\Pb\Big(\sup_{\tilde \tau\in [0,\varkappa]}\frac{x_N^{\rho_+}(\tilde c_2\tilde \tau T^{2/3})-x_N^{\rho_+}(0)-(1-\rho_+)\tilde c_2\tilde \tau T^{2/3}}{-\tilde c_1 T^{1/3}}>\tfrac14 \varkappa^2(\ell+1)^2-v\varkappa\Big),
\end{aligned}
\end{equation}
where in the last step we used the stationarity of $\widetilde B^{\rho_+}$, which is given in terms of $x_N^{\rho_+}$.

Now recall that $\ell\lesssim \e T^{1/3}$ for an $\e>0$ which we could take arbitrarily small (but independent of $T$). Thus, by setting $\e$ small enough and then $T$ large enough, we get
\begin{equation}
v\leq \frac{2\tilde c_2}{\tilde c_1\alpha_0^{1/3}}\Big(\kappa + \frac{2 (\sqrt{\alpha_0}-\sqrt{\alpha})^{2/3}\varkappa(\ell+1)}{\alpha^{1/3}}\Big)\leq \tfrac18 \varkappa (\ell+1)^2
\end{equation}
by choosing $\kappa=\varkappa (\ell+1)^2 \tilde c_1\alpha_0^{1/3}/(32\tilde c_2)$ and taking $\ell\geq \ell_0$ for some $\ell_0$ large enough (depending only on $\alpha,\alpha_0$). Thus
\begin{equation}
\eqref{eq1.106}\leq \Pb\Big(\sup_{\tilde \tau\in [0,\varkappa]}\frac{x_N^{\rho_+}(\tilde c_2\tilde \tau T^{2/3})-x_N^{\rho_+}(0)-(1-\rho_+)\tilde c_2\tilde \tau T^{2/3}}{-\tilde c_1 T^{1/3}}>\tfrac18 \varkappa^2(\ell+1)^2\Big).
\end{equation}
Set $W(\tau)=\frac{x_N^{\rho_+}(\tilde c_2\tilde \tau T^{2/3})-x_N^{\rho_+}(0)-(1-\rho_+)\tilde c_2\tilde \tau T^{2/3}}{-\tilde c_1 T^{1/3}}$ and $C=\tfrac18 \varkappa^2(\ell+1)^2$. Then $W(\tau)$ is a martingale and for $\lambda>0$, $e^{\lambda W(\tau)}$ a positive submartingale. Thus, for $\lambda>0$, we have
\begin{equation}
\Pb\Big(\max_{\tilde \tau\in [0,\varkappa]} W(\tau)\geq C\Big)= \Pb\Big(\max_{\tilde \tau\in [0,\varkappa]} e^{\lambda W(\tau)}\geq e^{\lambda C}\Big)
\leq \frac{\E\big(e^{\lambda W(\varkappa)}\big)}{e^{\lambda C}}.
\end{equation}
As this holds for any $\lambda>0$, we get
\begin{equation}
\Pb\Big(\max_{\tilde \tau\in [0,\varkappa]} W(\tau)\geq C\Big)\leq \inf_{\lambda>0}\frac{\E\big(e^{\lambda W(\varkappa)}\big)}{e^{\lambda C}}.
\end{equation}
The computation of $\E\big(e^{\lambda W(\varkappa)}\big)$ is elementary as $W(\varkappa)$ comes from centering and scaling a Poisson distribution. Optimizing over $\lambda>0$ one finally gets a Gaussian decay of the probability we are considering: for all $T$ large enough,
\begin{equation}
\eqref{eq1.106}\leq C e^{-c \ell^2 \varkappa}
\end{equation}
for some new constants $C,c>0$ (depending on $\alpha$ but not on $T$ and $\varkappa$). Plugging this into \eqref{eq1.103} with $\kappa$ chosen as just mentioned, we finally get
$\eqref{eq1.103}\leq C e^{-c \varkappa \ell^2}$ for some new constants $C,c>0$.

With this bound, we can go back to our estimate \eqref{eq1.90} and obtain the following: there exists an $\e_0>0$, an $\ell_0\in\N$ (depending only on $\alpha,\alpha_0$) such that for all $\e\in (0,\e_0]$
\begin{equation}
\Pb\Big(\sup_{\tau \in[\ell_0 \varkappa,\e \tilde c_2^{-1}T^{1/3}]} (\widetilde X_T(\tau)+\tau^2-g_T(\tau))>S\Big)\leq C e^{-c\varkappa\ell_0^2}
\end{equation}
uniformly for all $T$ large enough.
Similarly we get
\begin{equation}
\Pb\Big(\sup_{\tau \in[-\e \tilde c_2^{-1}T^{1/3},-\ell_0 \varkappa]} (\widetilde X_T(\tau)+\tau^2-g_T(\tau))>S\Big)\leq C e^{-c\varkappa\ell_0^2}
\end{equation}
uniformly for all $T$ large enough.
\end{proof}

\section{Second class particle in a TASEP with a moving
wall}\label{SecSecondClass}

In this section we consider a TASEP with one second class particle. It has the following initial configuration: all negative positions are filled by first class particles, position 0 is filled by a second class particle, and all positive positions are empty (holes). In addition, we consider a moving wall which starts at position $c T>0$ and moves to the right with speed $v\geq 0$ (as in Section~\ref{secConstantSpeed}), i.e., at time $t \in [0,T]$ it is at position $cT + v t$. All particles jump with rate 1, as before. We denote the position of the second class particle at time $t$ by $\mathfrak{f} (t)$.

Let $\mathrm{Unif}(a,b)$ be the uniform distribution on a segment $(a,b) \subset \R$, and denote by $a \delta(b)$ the atomic measure of weight $a$ at the point $b$. The main result of this section is the following.
\begin{thm}
\label{th:sec-class}
Assume that $0 \le v < 1$, and $c>0$. \\
(a) If $v+c \le 1$, then we have
\begin{multline}
\lim_{T \to \infty} \frac{\mathfrak{f} (T)}{T} \stackrel{(d)}{=} \frac{1}{2} \mathrm{Unif} \left( -1, -1+2v + 2 \sqrt{c (1-v)} \right) \\ + \left( 1 - v -\sqrt{c(1-v)} \right) \delta \left( -1+2v + 2 \sqrt{c(1-v)} \right).
\end{multline}
(b) If  $v+c \geq 1$, then we have
\begin{equation}
\lim_{T \to \infty} \frac{\mathfrak{f} (T)}{T}  \stackrel{(d)}{=} \frac{1}{2} \mathrm{Unif}(-1, 1).
\end{equation}
\end{thm}

\begin{remark}
The second case of Theorem~\ref{th:sec-class} corresponds to the situation in
which the evolving particles
are essentially not affected by the wall; thus, the limiting distribution is
the same as in the classical
result of Ferrari-Kipnis~\cite{FK95}. In the first case, we see that the second
class particle is ``repelled''
by the wall, which creates an atom at the shock point in the limiting
distribution. The macroscopic shock
position is easy to derive: denote by $\xi T$ the position of the shock at time
$t=T$. Since the speed of
the wall is $v$ and the speed of particles with density $\rho$ is $1-\rho$ (in
the law of large numbers approximation), the particle density in the
shock region is going to be $1-v$. Then $\xi$ is determined by the requirement
that the particles that would be to the
right of $\xi T$ in the system without wall, a total of $(1-\xi)^2T/4$ of
them, now form a constant density
region between $\xi T$ and $(v+c)T$, i.e., $(1-\xi)^2 T/4=(1-v)(v+c-\xi)T$.
This gives
$\xi=-1+2v + 2 \sqrt{c(1-v)}$.
\end{remark}

Theorem~\ref{th:sec-class} is proved in Section~\ref{sec:proof2cl} below.

\subsection{Distribution of the second class particle}\label{sec:2class-distrEq}
In this section we relate the distribution of the second class particle in the TASEP with a moving wall to an observable of the single-species TASEP.

Consider the TASEP with (first class) particles and holes only. Its initial configuration is given by $\tilde x_k (0) = -k+1, \ k \ge 1$ (step initial configuration). Assume that at time $t$ the jumps are allowed to happen only at positions $\ge (s- cT - v(T-t))$; the jumps to the \textit{left} of such a moving wall are suppressed. Here $s$ is an arbitrary integer.

\begin{prop}
One has
\begin{equation}
\label{eq:interm2cl}
\Pb\left( \mathfrak{f} (T) \ge s \right) = \Pb\left( \mbox{there exists $k \ge 1$ such that $\tilde x_k (T)=s$} \right)
\end{equation}
\end{prop}
\begin{proof}

For the proof, we will use multi-species processes introduced in Section~\ref{proofStartingIdentity}, with a particular choice of the function $f(t) := cT+ vt$. We couple the process $\{\hat \eta_{t;f} (z) \}$ with a TASEP with one second class particle via identifying all particles with negative colors with the first class particles, the particle of color 0 with the second class particle, and all particles with positive colors with holes. Using Proposition~\ref{prop:inverTime} and this coupling, we obtain
\begin{equation}
\label{eq:2classInvers}
\Pb\left( \mathfrak{f} (T) \ge s \right) = \Pb\left( \eta_{T;f,T} (0) \ge s \right).
\end{equation}
Recall that in the process $\eta_{T;f,T}$ the wall at time $t \in [0,T]$ is at position $cT + v (T-t)$. In the right-hand side, we identify colors $\ge s$ with holes, and $<s$ with the first class particles, do the particle-hole involution and shift the coordinate axis by $s$. This provides a coupling with the process $\{ \tilde x_k (t) \}$ and completes the proof of the proposition.
\end{proof}

\newpage
\subsection{Asymptotics}\label{sec:proof2cl}
In this section we provide a proof of Theorem~\ref{th:sec-class}.
\begin{proof}[Proof of Theorem~\ref{th:sec-class}]
First, let us relate the process $\tilde x_k (t)$ from Section~\ref{sec:2class-distrEq} with the standard (without a moving wall) single-species TASEP $x_k(t)$ which starts from the same step initial configuration. Let us use the basic coupling for these two processes. We claim that if for a fixed $k \ge 1$ one has
\begin{equation}
x_k(t) \ge s- cT - v(T-t)  \mbox{ for all $t \in [0,T]$,}
\end{equation}
then $x_k (t) = \tilde x_k (t)$, for all $t$. Indeed, in this case the particle position $\tilde x_k (t)$ depends only on the non-suppressed jumps, and they coincide for both processes. In the opposite case, let
\begin{equation}
\tau_k:= \sup_{\tau \in \R_{\ge 0}} \left\{ x_k(t) \ge s- cT - v(T-t) \mbox{ for all $t \in [0,\tau]$}  \right\}.
\end{equation}
In words, $\tau_k$ is the moment when the particle $x_k (t)$ is ``caught'' by the wall moving from the left, and we have
\begin{equation}
\label{eq:caughtWall}
\tilde x_k (t) = s- cT - v(T-\tau_k) \mbox{ for all $t \in [\tau_k,T]$},
\end{equation}
since the particle $\tilde x_k (t)$ does not move after that moment. Note that all particles are split in two groups: There exists a (random) integer $L$ such that the particles $ \tilde x_1,  \tilde x_2, \dots, \tilde x_L$ are not ''caught'' by the wall and thus coincide with $x_1, x_2, \dots, x_L$ at all moments of time, while all particles with larger than $L$ labels satisfy \eqref{eq:caughtWall}.

Our goal is to analyse the limit behavior of the right-hand side of equation \eqref{eq:interm2cl}. Note that the particles in the process $\{ \tilde x_k \}$ which satisfy \eqref{eq:caughtWall} cannot occupy position $s$ since the wall is always to the left of $s$ during the time interval $[0,T]$. Thus, we have $\tilde x_k (T) = s$ if and only if
\begin{equation}
x_k (T) = s, \quad \mbox{and} \quad x_k(t) \ge s- cT - v(T-t) \mbox{ for all $t \in [0,T]$}.
\end{equation}
By the hydrodynamical (law of large numbers) limit for the TASEP with step initial condition (first obtained in~\cite{R81}), we have
\begin{equation}
\label{eq:2clHydro}
\lim_{T \to \infty} \Pb \left( \mbox{there exists $k \ge 1$ such that $x_k (T) = \lfloor \hat s T \rfloor$} \right) = \frac{1-\hat s}{2}, \qquad \hat s \in [-1,1].
\end{equation}
Set
\begin{equation}
\mathfrak{a}:= \left( \frac{1-\hat s}{2} \right)^2.
\end{equation}
Note that by \eqref{eqLLN} one has $x_{\lfloor \mathfrak{a} T \rfloor} (T) \approx \hat s T$ (here and below we use $\approx$ for denoting the asymptotic equivalence in the $T \to \infty$ limit). Furthermore, by \eqref{eqLLN} one has
\begin{equation}
x_{\lfloor \mathfrak{a} T \rfloor} (\alpha_0 T) \approx \sqrt{\alpha_0} \left( \sqrt{\alpha_0} - 1 + \hat s \right) T, \qquad \mbox{for $\alpha_0 \in [\mathfrak{a},1]$}.
\end{equation}
In words, with non-negligible probability the position $\lfloor \hat s T \rfloor$ can be occupied only by a particle with a number that asymptotically behaves as $\mathfrak{a} T$.
The asymptotic behavior of \eqref{eq:interm2cl} depends on whether the particles with such numbers were caught by the moving wall or not.

Let us consider two cases.

\medskip

\textbf{Case 1:} Assume that $-1 < \hat s < -1+2v + 2 \sqrt{c (1-v)}<1$. By straightforward calculus, one has
\begin{equation}
\alpha_0 - (1-\hat s) \sqrt{\alpha_0} > \hat s - c - v (1-\alpha_0) \quad \mbox{for all }\alpha_0 \in [\mathfrak{a},1].
\end{equation}
This implies that for such values of parameters there exist $\eps, \delta>0$ such that
\begin{equation}
\Pb\left(  x_{ \lfloor (\mathfrak{a}+\eps)T \rfloor} ( \alpha_0 T) > (\hat s T - cT - v(T-\alpha_0 T)) + \delta T \mbox{ for all }\alpha_0 \in [0,1] \right) = 1 -o(1).
\end{equation}
Analogously to Lemma~\ref{lemmaMacroPoints}, one can extend this to a uniform in time estimate
\begin{equation}
\Pb\left( x_{\lfloor (\mathfrak{a}+\eps)T \rfloor} (t) > (\hat s T - cT - v(T-t)) + \delta T \mbox{ for all }t \in [0,T] \right) = 1 -o(1).
\end{equation}
Therefore, the particle $x_{\lfloor (\mathfrak{a}+\eps)T \rfloor}$ is not caught by the moving wall with probability close to 1, which implies that all particles with numbers $\approx \mathfrak{a} T$ will not be caught by the moving wall  with probability close to 1. Thus, the limit of the probability in the right-hand side of \eqref{eq:interm2cl} is given by \eqref{eq:2clHydro}. We obtain
that for $-1 < \hat s < -1+2v + 2 \sqrt{c (1-v)}<1$, one has
\begin{equation}
\lim_{T \to \infty} \Pb\left( \frac{\mathfrak{f} (T)}{T} \ge \hat s \right) = \frac{1-\hat s}{2}.
\end{equation}

\medskip

\textbf{Case 2:} Assume that $1 > \hat s > -1+2v + 2 \sqrt{c (1-v)}> -1$. Then there exists $\alpha_0 \in [\mathfrak{a},1)$ such that
\begin{equation}
\alpha_0 - (1-\hat s) \sqrt{\alpha_0} < \hat s - c - v (1-\alpha_0).
\end{equation}
By the law of large numbers, this implies that for a sufficiently small $\eps>0$ the particle $x_{(\mathfrak{a} - \eps) T}$ will be caught by the moving wall with probability close to 1, which implies that all particles with numbers $\approx \mathfrak{a} T$ will be caught by the wall with overwhelming probability. Therefore, one has
\begin{equation}
\lim_{T \to \infty} \Pb\left( \frac{\mathfrak{f} (T)}{T} \ge \hat s \right) = 0,
\end{equation}
in this case.

This concludes the proof of Theorem~\ref{th:sec-class}.
\end{proof}

\newpage
\appendix

\section{Known results and estimates}

Let us start with a well-known result on the convergence of Poisson processes to the Brownian motion.
\begin{lem}\label{LemPPtoBM}
Let $Z(t)$ be the number of particles in $[0,t]$ of a Poisson process with intensity $\lambda$. Then
\begin{equation}
\lim_{t\to\infty}\frac{Z(\tau t)-\lambda\tau t}{\sqrt{t}} = B(\lambda \tau),
\end{equation}
in the sense of weak convergence in the sup-norm on finite intervals, where $B$ is a standard Brownian motion.
\end{lem}

To localize the starting points of the backwards paths we use estimates for the one-point distribution of tagged particles for step and stationary initial conditions.
\begin{lem}\label{LemOnePointBounds}
For any $\alpha\in (0,1)$,
\begin{equation}\label{eqOnePointLimitLaw}
\lim_{T\to\infty}\Pb(x_{\alpha t}(t)\geq (1-2\sqrt{\alpha})t-s c_1(\alpha) t^{1/3}) = F_{\rm GUE}(s),
\end{equation}
with $c_1(\alpha)=\frac{(1-\sqrt{\alpha})^{2/3}}{\alpha^{1/6}}$. There exists constants $C,c>0$ such that, uniformly for all $t$ large enough,
\begin{equation}\label{eqOnePointUpperTail}
\Pb(x_{\alpha t}(t)\leq (1-2\sqrt{\alpha})t-s c_1(\alpha) t^{1/3}) \leq  C e^{-c s},\quad s>0,
\end{equation}
and
\begin{equation}\label{eqOnePointLowerTail}
\Pb(x_{\alpha t}(t)\geq (1-2\sqrt{\alpha})t-s c_1(\alpha) t^{1/3}) \leq  C e^{-c |s|^{3/2}},\quad -o(t^{2/3})\lesssim s<0.
\end{equation}
The constants $C,c$ can be chosen uniformly for $\alpha$ in a closed subset of $(0,1)$.
\end{lem}
Similar statements have been derived in the framework of the (directed) last passage percolation in a quadrant with exponential weights. If we denote by $L_{m,n}$ the last passage time from the origin to the point $(m,n)$, then to translate the results to the TASEP particle positions we use the relation
\begin{equation}\label{eq4}
\Pb(x_n(t)\geq m-n)=\Pb(L_{m,n}\leq t).
\end{equation}
\eqref{eqOnePointLimitLaw} was proven in Theorem~1.6 of~\cite{Jo00b}. Since the distribution function \eqref{eq4} is given by a Fredholm determinant, the upper tail \eqref{eqOnePointUpperTail} is easily obtained from the exponential tail of the kernel either in the TASEP representation, or in the LPP representation. For a statement for the upper tail in LPP one can go back for instance to the work on Laguerre ensembles~\cite{BBP06}; for explicit statements on the tails see e.g.\ Section~4.1 of~\cite{FN13}, Lemma~1 of~\cite{BFP12}. The lower tail in LPP was proven in~\cite{BFP12} (Proposition~3 in combination with (56)). Applying \eqref{eq4} with $n=\alpha t$ and $m=(1-\sqrt{\alpha})^2 t-s\tilde c_1(\alpha) t^{1/3}$ we get the result. The condition $-s=o(t^{2/3})$ is to ensure that $\eta=n/m$ stays bounded away from $0$ and $\infty$. Presumably, using the approach of~\cite{BFP12} directly with the kernel of TASEP particles, this restriction would not appear. However since the bound is sufficient for our purpose, we do not investigate this further.

\begin{lem}\label{LemOnePointStat}
Consider the stationary TASEP with density $\rho$, with right-to-left labeling $x_{n+1}(t)<x_n(t)$ such that at time $t=0$, $x_1(0)<0\leq x_0(0)$.
Let $N=\rho^2 t-2 w \rho \chi^{1/3} t^{2/3}$, with $\chi=\rho(1-\rho)$. Then
\begin{equation}\label{eqOnePointLimitLawStat}
\lim_{t\to\infty} \Pb\left(x^{\rho}_N(t)\geq (1-2\rho)+2 w \chi^{1/3} t^{2/3}-(1-\rho)\chi^{-1}s t^{1/3}\right)= F_{{\rm BR},w}(s),
\end{equation}
where $F_{{\rm BR},w}$ is the Baik-Rains distribution function with parameter $w$. Furthermore, there exists constants $C,c>0$ such that
\begin{equation}\label{eqOnePointUpperTailStat}
\Pb(x^{\rho}_N(t)\leq (1-2\rho)+2 w \chi^{1/3} t^{2/3}-(1-\rho)\chi^{-1}s t^{1/3}) \leq  C e^{-c s},\quad s>0,
\end{equation}
and
\begin{equation}\label{eqOnePointLowerTailStat}
\Pb(x^{\rho}_N(t)\geq (1-2\rho)+2 w \chi^{1/3} t^{2/3}-(1-\rho)\chi^{-1}s t^{1/3}) \leq C e^{-c |s|^{3/2}},\quad -o(t^{2/3})\lesssim s<0
\end{equation}
uniformly for all $t$ large enough. The constants $C,c$ can be chosen uniformly for $\rho$ in a closed subset of $(0,1)$.
\end{lem}
Again the results have been proven already in the last passage percolation framework and mapped back to TASEP. \eqref{eqOnePointLimitLawStat} was proven in Theorem~1.6 of~\cite{BFP09}. \eqref{eqOnePointLowerTailStat} follows from \eqref{eqOnePointLowerTail}, because if we couple the step and stationary initial condition through the basic coupling, then $x_N(t)\geq x_N^\rho(t)$. Thus
\begin{equation}
\Pb(x_N^\rho(t)\geq x)\leq \Pb(x_N(t)\geq x).
\end{equation}
With the choice $x=(1-2\rho)+2 w \chi^{1/3} t^{2/3}-(1-\rho)\chi^{-1}s t^{1/3}$ and $N=\rho^2 t-2 w \rho \chi^{1/3} t^{2/3}$, this corresponds to taking in Lemma~\ref{eqOnePointLimitLaw}, $\alpha=\rho^2-2w\rho \chi^{1/3} t^{-1/3}$ and $s\to w^2+s/\chi^{2/3}$. The shift of $w^2$ and the scaling $\chi^{2/3}$ just lead to different constants.

Finally, to get \eqref{eqOnePointUpperTailStat} one can do the following: consider $x^{\rho,\rm right}$ starting with density $1$ on $\Z_-$ and density $\rho$ (Bernoulli) on $\N$ and $x^{\rho,\rm left}$ starting with density $\rho$ on $\Z_-$ and empty on $\N$. Then $x_n^\rho(t)=\min\{x_n^{\rho,\rm left}(t),x_n^{\rho,\rm right}(t)\}$, cf.\ \eqref{eqLeftRightDecomposition}, so that \mbox{$\Pb(x_n^\rho(t)\leq x)\leq \Pb(x_n^{\rho,\rm left}(t)\leq x)+\Pb(x_n^{\rho,\rm right}(t)\leq x)$}. These last two distribution functions can be written as Fredholm determinants, and the analysis of the kernels would give the desired result. For those initial conditions, the bounds for the last passage percolation model have been already proven by using the estimates of the kernels provided in~\cite{BBP06}, see Lemma~3.3 of~\cite{FO17}.


\end{document}